\documentclass[amssymb,12pt ]{article}
\pdfoutput=1
\usepackage{geometry}
\usepackage{amsfonts}
\usepackage{amssymb}
\usepackage{amsmath}
\usepackage{amsthm}
\usepackage{graphicx}
\usepackage{picture}
\usepackage{curves}
\usepackage{subfigure}
\usepackage{epic}
\usepackage{color}

\newtheorem{thm}{Theorem}[section]
\newtheorem{cor}[thm]{Corollary}
\newtheorem{lem}[thm]{Lemma}
\newtheorem{pro}[thm]{Proposition}

\newtheorem{obs}[thm]{Observation}
\newenvironment{ack}{\noindent{\bf Acknowledgments}}

\newcommand{\vol}{{\rm vol}}
\newcommand{\cs}{{\rm cs}}
\newcommand{\li}{{\rm Li}_2}

\newcommand{\olim}{\underset{N\rightarrow\infty}{\text{\rm o-lim}}}
\newcommand{\modulo}{~~({\rm mod}~\pi^2)}
\newcommand{\modulos}{~~({\rm mod}~4\pi^2)}

\begin{document}

\title{Optimistic limits of Kashaev invariants and complex volumes of hyperbolic links
}
\author{\sc Jinseok Cho, Hyuk Kim and Seonhwa Kim}
\maketitle
\begin{abstract}
Yokota suggested an optimistic limit method of the Kashaev invariants of hyperbolic knots
and showed it determines the complex volumes of the knots. His method is very effective and
gives almost combinatorial method of calculating the complex volumes.
However, to describe the triangulation of the knot complement, 
he restricted his method to knot diagrams with certain conditions.
Although these restrictions are general enough for any hyperbolic knots, 
we have to select a good diagram of the knot to apply his theory.
%they make his method hard to apply for some cases.

In this article, we suggest more combinatorial way to calculate the complex volumes of hyperbolic links
using the modified optimistic limit method.
%which is based on the parameter space of arbitrary ideally triangulated manifolds suggested by Luo.
This new method works for any link diagrams, and it is more intuitive, easy to handle and has natural geometric meaning.
\end{abstract}

\section{Introduction}\label{sec1}

Kashaev conjectured the following relation in \cite{Kashaev97} :
  $$2\pi \lim_{N\rightarrow\infty}\frac{\log|\langle L\rangle_N|}{N}=\vol(L),$$
  where $L$ is a hyperbolic link, vol($L$) is the hyperbolic volume of $\mathbb{S}^3\backslash L$,
  and $\langle L\rangle_N$ is the $N$-th Kashaev invariant of $L$.
  After that, the generalized conjecture was proposed in \cite{Murakami02} that
 $$2\pi i\lim_{N\rightarrow\infty}\frac{\log\langle L\rangle_N}{N}\equiv i(\vol(L)+i\,\cs(L))\modulo,$$
 where cs($L$) is the Chern-Simons invariant of $\mathbb{S}^3\backslash L$ defined modulo $\pi^2$ in \cite{Meyerhoff86}.
These are now called {\it Kashaev volume conjectures} and $\vol(L)+i\,\cs(L)$ is called {\it the complex volume} of $L$.

When Kashaev suggested the conjecture in \cite{Kashaev97}, he calculated a certain value using
an analytic function induced from the Kashaev invariant
and showed numerically that the value coincides with the volume of the link for  a few cases,  
 where the function was the same  one called {\it the potential function} in \cite{Neumann95}.  
After that,  the value was named {\it the optimistic limit} of the Kashaev invariant and denoted by $2\pi i\,\olim\frac{\log\langle L\rangle_N}{N}$ in \cite{Murakami00b}.
(See Section 1.1 of \cite{Cho13b} for the exact meaning of the optimistic limit.)

 The potential function is closely related to the quantum dilogarithm function of Faddeev as in \cite{Kashaev97},
especially, it can be considered as a classical limit of the partition function defined in \cite{Hikami07},
where the partition function is defined by assigning the quantum dilogarithms to each hyperbolic ideal tetrahedra
and integrating the product of them. Quantum dilogarithm satisfies the 3--2 Pachner move, so the partition function
is expected to be independent of the chosen triangulation. Similar ideas were used in \cite{Baseilhac07} and \cite{Kashaev14}
to define quantum (hyperbolic) field theories and certain invariants of knots incompact oriented 3-manifolds using partition functions, which suggests
the potential functions can be used, not only in the optimistic limits, but also in many other situations.

 Regarding the optimistic limit, Yokota proved
$$2\pi i\,\olim\frac{\log\langle K\rangle_N}{N} \equiv i(\vol(K)+i\,\cs(K))\modulo,$$
for hyperbolic knots $K$ in \cite{Yokota10} by introducing natural geometry corresponding to the optimistic limit.
Elaborating on the geometry, he defined a triangulation of $\mathbb{S}^3\backslash (K\cup\{\text{two points}\})$ 
and transformed it into the triangulation of $\mathbb{S}^3\backslash K$ by collapsing certain tetrahedra. 
He defined the potential function reflecting this collapsing process and proved the derivation of this function
gives the hyperbolicity equations, i.e. Thurston's gluing equations and the completeness conditions of the triangulation.
(See Section \ref{sec3} for the definitions.)
His method is very effective and
gives almost combinatorial method of calculating the complex volumes of hyperbolic knots.
(See \cite{Cho10b} for a brief survey.)

However, understanding Yokota's method in \cite{Yokota10} is not easy for several reasons. We think
a major difficulty lies on the collapsing process of the triangulation.
To make the collapsing works well, he deformed the knot diagram into certain (1,1)-tangle diagram 
satisfying several non-trivial conditions and restricted his method only to knots.
%(For example, his theory is not applicable to the trefoil knot diagram in Figure \ref{3_1}.)
Furthermore, the collapsing process twists the natural triangulation to a complicate one.
To overcome these difficulties, we will develop new version of Yokota theory without collapsing process here.
Our method does not need to deform the diagram because it is applicable to any link diagrams without restriction.

In Section \ref{sec2} of this article, 
we define the natural potential function $V(z_1,\ldots,z_n)$ of a hyperbolic link $L$ combinatorially from the link diagram.
Then we will consider the following set of equations
\begin{equation}\label{defH}
\mathcal{H}:=\left\{\left.\exp(z_k\frac{\partial V}{\partial z_k})=1\right|k=1,\ldots,n\right\}.
\end{equation}

In Section \ref{sec3}, we introduce an ideal triangulation of $\mathbb{S}^3\backslash (L\cup\{\text{two points}\})$,
and name it {\it octahedral triangulation}. It was the same one considered in \cite{Yokota10} before the collapsing, and it also appeared
in \cite{Weeks05} as a natural triangulation of the link complement inside $\mathbb{S}^2\times [0,1]$.
On the other hand, Luo considered ideal triangulations of closed 3-manifolds by removing vertices in \cite{Luo13} and 
considered their hyperbolicity equations. Later, Luo, Tillmann and many others considered ideal triangulations of any 3-manifolds
by removing non-ideal vertices and found several properties of their hyperbolicity equations. (See \cite{Tillmann11} for example.)
We consider the hyperbolicity equations of the octahedral triangulation in this sense. 
Note that this ideal triangulation of $\mathbb{S}^3\backslash (L\cup\{\text{two points}\})$ can be
obtained by removing two non-ideal points from the triangulation of $\mathbb{S}^3\backslash L$, as in \cite{Tillmann11}.

One of the most important properties of the potential function $V$ is the following proposition.

\begin{pro}\label{pro1} For a hyperbolic link $L$ with a fixed diagram, consider the potential function $V(z_1,\ldots,z_n)$ defined in Section \ref{sec2}.
Then the set $\mathcal{H}$ defined in (\ref{defH}) becomes 
the hyperbolicity equations of the octahedral triangulation of $\mathbb{S}^3\backslash (L\cup\{\text{two points}\})$.
\end{pro}

The exact construction of the triangulation and the proof will be in Section \ref{sec3}.
We remark that this proposition also holds for the potential functions of the collapsed cases in \cite{Yokota10} and \cite{Cho13b},
but the proof of this article is more natural and far easier than the previous ones. 
This is because the collapsing process distorts the natural geometry of the triangulation,
so one has to keep track of the changes carefully.

Let $\mathcal{S}=\{(z_1,\ldots,z_n)\}$ be the set of solutions\footnote{
We only consider solutions satisfying the condition that, 
when the potential function is expressed by $V(z_1,\ldots,z_n)=\sum\pm\li(\frac{z_a}{z_b})$,
the variables inside the dilogarithms satisfy $\frac{z_a}{z_b}\notin\{0,1,\infty\}$. 
Previously, in \cite{Yokota10} and \cite{Cho13b}, these solutions were called essential solutions.}
of $\mathcal{H}$ in $\mathbb{C}^n$. In this article, we always assume $\mathcal{S}\neq \emptyset$.
Then, by Theorem 1 of \cite{Tillmann11}, all edges in the octahedral triangulation are essential.
(Essential edge roughly means it is not null-homotopic. See \cite{Tillmann11} for the exact definition.)
Therefore, using Yoshida's construction in Section 4.5 of \cite{Tillmann13}, for a solution $\bold{z}\in \mathcal{S}$, 
we can obtain the boundary-parabolic representation\footnote{
The solution $\bold{z}\in\mathcal{S}$ satisfies the completeness condition, so $\rho_{\bold{z}}$ is boundary-parabolic.}
\begin{equation}\label{repre}
\rho_{\bold{z}}:\pi_1(\mathbb{S}^3\backslash L)\longrightarrow{\rm PSL}(2,\mathbb{C}).\end{equation}
Note that the volume $\vol(\rho_{\bold{z}})$ and the Chern-Simons invariant $\cs(\rho_{\bold{z}})$
of $\rho_{\bold{z}}$ were defined in \cite{Zickert09}.
We call $\vol(\rho_{\bold{z}})+i\,\cs(\rho_{\bold{z}})$ {\it the complex volume of} $\rho_{\bold{z}}$.

For the solution set $\mathcal{S}$, let $\mathcal{S}_j$ be a path component of $\mathcal{S}$ 
satisfying $\mathcal{S}=\cup_{j\in J}\mathcal{S}_j$ for some index set $J$. We assume $0\in J$ for notational convenience. 
To obtain well-defined values of the potential function $V(z_1,\ldots,z_n)$ (see Lemma \ref{branch}), we slightly modify it  to
\begin{equation}\label{defV_0}
V_0(z_1,\ldots,z_n):=V(z_1,\ldots,z_n)-\sum_{k=1}^n \left(z_k\frac{\partial V}{\partial z_k}\right)\log z_k.
\end{equation}
Then the main result of this article is as follows:

\begin{thm}\label{thm1}
Let $L$ be a hyperbolic link with a fixed diagram and  $V(z_1,\ldots,z_n)$ be the potential function of the diagram. 
Assume the solution set $\mathcal{S}=\cup_{j\in J}\mathcal{S}_j$ is not empty.
Then, for any ${\bold z}\in \mathcal{S}_j$, $V_0({\bold z})$ is constant (depends only on $j$) and
\begin{equation}\label{V1}
V_0(\bold{z})\equiv i\,(\vol(\rho_{\bold z})+i\,\cs(\rho_{\bold z}))\modulo,
\end{equation}
where $\rho_{\bold z}$ is the boundary-parabolic representation obtained in (\ref{repre}).
Furthermore, there exists a path component $\mathcal{S}_0$ of $\mathcal{S}$ satisfying
\begin{equation}\label{V2}
V_0(\bold{z_{\infty}})\equiv i\,(\vol(L)+i\,\cs(L))\modulo,
\end{equation}
for all ${\bold z}_{\infty}\in \mathcal{S}_0$.
\end{thm}
The proof will be given in Section \ref{sec4}. The main idea is to use
Zickert's formula of the extended Bloch group in \cite{Zickert09}, which was already appeared in \cite{Yokota10}. 
 However, our proof is simpler because we do not consider any collapsing. 
We call the value $V_0({\bold z})$ {\it the optimistic limit of the Kashaev invariant}
 and note that it depends on the choice of the diagram and the path component $\mathcal{S}_j$.
Finally, in Section \ref{sec5}, we apply our results to the twist knots and calculate the complex volumes of representations.

%\begin{remark}
	Although we restricted $L$ to hyperbolic links, Proposition \ref{pro1} and Theorem \ref{thm1} 
 still hold for   non-hyperbolic links \footnote{In this case, we need a minor assumption that no component of the link diagram has only over-crossings or only under-crossings.} except for the existence of $S_0$ and (\ref{V2}).
(The definitions of $\vol(\rho_{\bold z})$ and $\cs(\rho_{\bold z})$ are from \cite{Zickert09}.)
That is because we do not use the hyperbolic structure of $L$ 
but the boundary-parabolic representation $\rho_{\bold z}$ in (\ref{repre}), which can be
non-discrete or non-faithful.
%\end{remark}

%This article consists of the following contents. In Section \ref{sec2}, we define the potential function $V(z_1,\ldots,z_n)$ from the link diagram.
%In Section \ref{sec3}, we introduce the octahedral triangulation of $\mathbb{S}^3\backslash (L\cup\{\text{two points}\})$ and prove Proposition \ref{pro1}.
%Section \ref{sec4} is devoted to the proof of Theorem \ref{thm1}, and we apply our results to the twist knots 
%and calculate the complex volumes in Section \ref{sec5}.

%\begin{remark}  

 Finally, we remark that the following relation 
$$J_L(N;\exp\frac{2\pi i}{N})=\langle L \rangle_N$$
was proved in \cite{Murakami01a}, where $J_L(N;x)$ is the $N$-th colored Jones polynomial of $L$ with a complex variable $x$.
Therefore, it is very natural to consider the optimistic limit of the colored Jones polynomial,
and it will be discussed in the first author's another article \cite{Cho13c}.
%\end{remark}

\section{Potential function $V(z_1,\ldots,z_n)$}\label{sec2}

Consider a hyperbolic link $L$ and its diagram $D$. For simplicity,
we always assume $D$ does not have any kink by removing them as in Figure \ref{kink}.
\begin{figure}[h]
\centering
  \includegraphics[scale=0.3]{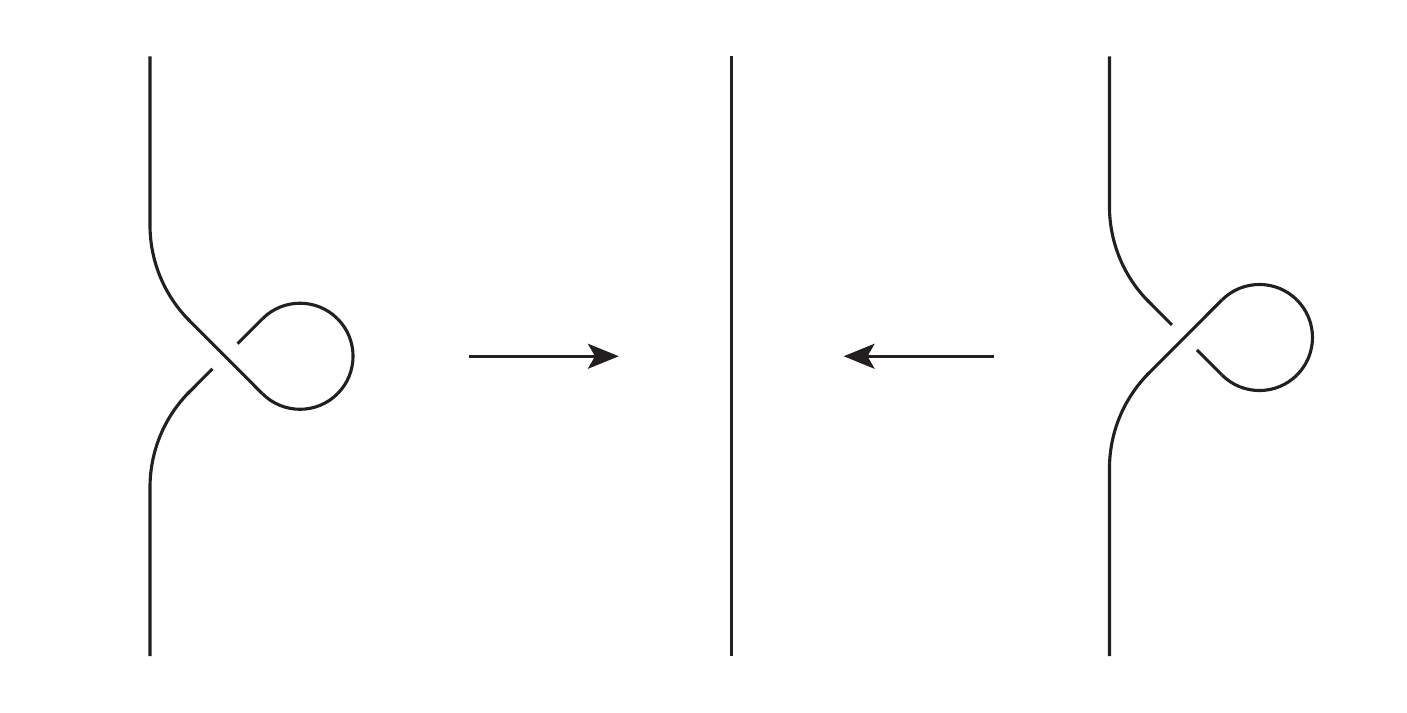}
  \caption{Removing kinks}\label{kink}
\end{figure}

We define {\it sides} of $D$ by the arcs connecting two adjacent crossing points.\footnote{
Most people use the word {\it edge} instead of {\it side} we are using here.
However, in this paper, we want to keep the word {\it edge} for the edge of a tetrahedron.}
For example, the diagram of the figure-eight knot $4_1$ in Figure \ref{4_1} has 8 sides.

\begin{figure}[h]
\centering
  \includegraphics[scale=0.6]{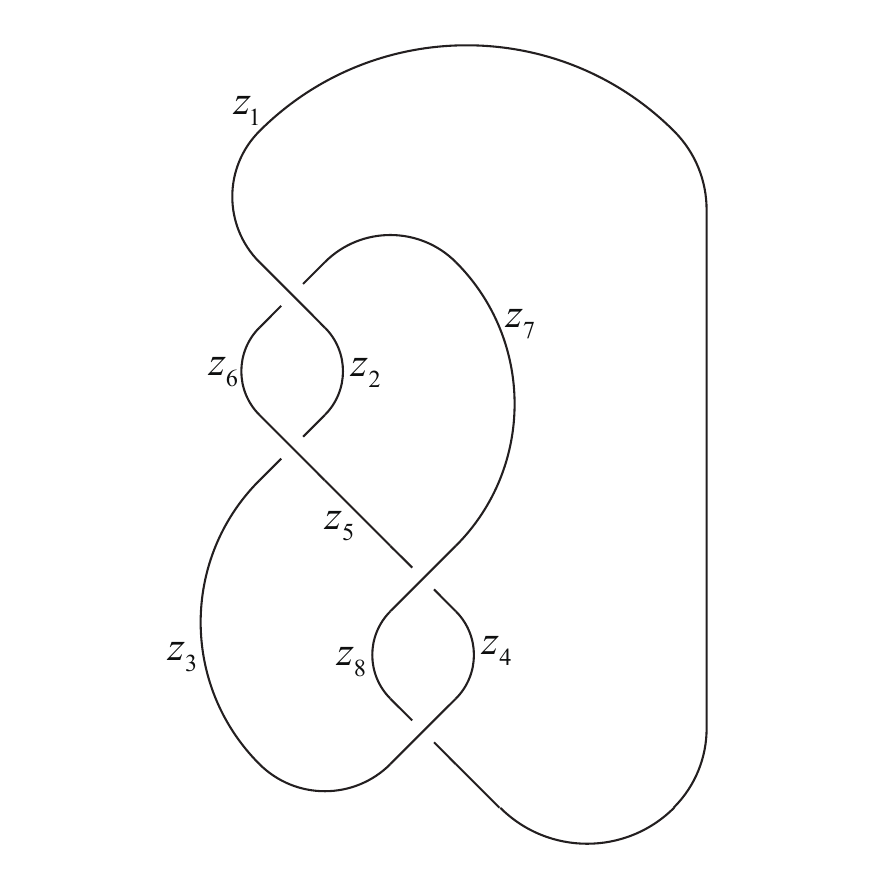}
  \caption{The figure-eight knot $4_1$}\label{4_1}
\end{figure}

We assign complex variables $z_1,\ldots,z_n$ to each side of the diagram $D$.
Using the dilogarithm function $\li(z)=-\int_0^z \frac{\log(1-t)}{t}dt$,
we define the potential function of a crossing as in Figure \ref{pic2}.
%(We always assume the principal branch of the logarithm, $-\pi<\imaginary\log z\leq \pi$.)

\begin{figure}[h]\centering
{\setlength{\unitlength}{0.4cm}
  \begin{picture}(30,6)\thicklines
    \put(6,5){\line(-1,-1){4}}
    \put(2,5){\line(1,-1){1.8}}
    \put(4.2,2.8){\line(1,-1){1.8}}
%    \put(3.5,1){$\frac{z_b}{z_a}$}
%    \put(5.5,3){$\frac{z_c}{z_b}$}
%    \put(3.5,4.5){$\frac{z_d}{z_c}$}
%    \put(1.5,3){$\frac{z_a}{z_d}$}
    \put(1,5.3){$z_d$}
    \put(6,5.3){$z_c$}
    \put(1,0.2){$z_a$}
    \put(6,0.2){$z_b$}
    \put(7,3){$\displaystyle\longrightarrow ~~\li(\frac{z_b}{z_a})-\li(\frac{z_b}{z_c})+\li(\frac{z_d}{z_c})-\li(\frac{z_d}{z_a})$}
  \end{picture}}
  \caption{Potential function of a crossing}\label{pic2}
\end{figure}
Note that the potential function in Figure \ref{pic2} comes from the formal substitution of the R-matrix of the Kashaev invariant in \cite{YokotaPre}. (See \cite{Cho13b} for the meaning of the formal substitution.)
In \cite{Cho10b}, we defined the potential function of the corner of a crossing from Figure \ref{corner}.
Following this definition, the potential function of a crossing is then the summation 
of potential functions of the four corners.

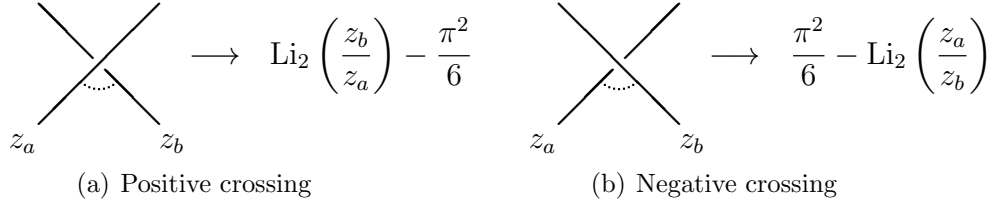
\begin{figure}[h]
\centering
  \subfigure[Positive crossing]
  { {\setlength{\unitlength}{0.4cm}
  \begin{picture}(15,6)\thicklines
    \put(6,5){\line(-1,-1){4}}
    \put(6,1){\line(-1,1){1.8}}
    \put(3.8,3.2){\line(-1,1){1.8}}
    \put(4,3){\arc[2](-0.6,-0.6){90}}
%    \put(3.5,1){$IT_l$}
    \put(1,0.2){$z_a$}
    \put(6,0.2){$z_b$}
    \put(7,3){$\displaystyle\longrightarrow ~~\li\left(\frac{z_b}{z_a}\right)-\frac{\pi^2}{6}$}
  \end{picture}}}\hspace{0.5cm}
  \subfigure[Negative crossing]
  { {\setlength{\unitlength}{0.4cm}
  \begin{picture}(15,6)\thicklines
    \put(2,5){\line(1,-1){4}}
    \put(2,1){\line(1,1){1.8}}
    \put(4.2,3.2){\line(1,1){1.8}}
    \put(4,3){\arc[2](-0.6,-0.6){90}}
%    \put(3.5,1){$IT_l$}
    \put(1,0.2){$z_a$}
    \put(6,0.2){$z_b$}
    \put(7,3){$\displaystyle\longrightarrow ~~\frac{\pi^2}{6}-\li\left(\frac{z_a}{z_b}\right)$}
  \end{picture}}}
  \caption{Potential function of a corner}\label{corner}
\end{figure}

{\it The potential function} $V(z_1,\ldots,z_n)$ of the diagram $D$ is defined by the summation of
all potential functions of the crossings. For example, the potential function of the figure-eight knot $4_1$
in Figure \ref{4_1} is
\begin{eqnarray*}
V(z_1,\ldots,z_8)=\left\{\li(\frac{z_6}{z_1})-\li(\frac{z_6}{z_2})+\li(\frac{z_7}{z_2})-\li(\frac{z_7}{z_1})\right\}\\
+\left\{\li(\frac{z_3}{z_6})-\li(\frac{z_3}{z_5})+\li(\frac{z_2}{z_5})-\li(\frac{z_2}{z_6})\right\}\\
+\left\{\li(\frac{z_4}{z_8})-\li(\frac{z_4}{z_7})+\li(\frac{z_5}{z_7})-\li(\frac{z_5}{z_8})\right\}\\
+\left\{\li(\frac{z_1}{z_3})-\li(\frac{z_1}{z_4})+\li(\frac{z_8}{z_4})-\li(\frac{z_8}{z_3})\right\}.
\end{eqnarray*}

We define a modified potential function $V_0(z_1,\ldots,z_n)$ as given in (\ref{defV_0}). 
Note that $V_0$ is analytic since the dilogarithm function $\li(z)$ is analytic and the term $z_k\frac{\partial V}{\partial z_k}$
consists of logarithms.
This property will be used implicitly in Lemma \ref{lem1} below.

Recall that $\mathcal{H}$ was defined in (\ref{defH}). 
Also recall that we are considering the solutions ${\bold z}=(z_1,\ldots,z_n)\in\mathbb{C}^n$ of $\mathcal{H}$ with the property
that if the potential function is expressed by $V(z_1,\ldots,z_n)=\sum\pm\li(\frac{z_a}{z_b})$,
then variables inside the dilogarithms satisfy $\frac{z_a}{z_b}\notin\{0,1,\infty\}$.
This choice is reasonable because, if $\frac{z_a}{z_b}\in\{0,1,\infty\}$ for some solution, then at least one of the terms 
$\frac{\partial V}{\partial z_a}$ and $\frac{\partial V}{\partial z_b}$ of $V_0(z_1,\ldots,z_n)$
is not well-defined at that solution.

In this article, we always assume the solution set $\mathcal{S}\subset\mathbb{C}^n$ of $\mathcal{H}$ is nonempty.
We cannot guarantee $\mathcal{S}\neq\emptyset$ for any link diagram. 
For example, the link diagrams containing Figure \ref{nosolution} always satisfy $\mathcal{S}=\emptyset$
because $\exp(z_4\frac{\partial V}{\partial z_4})=1$ implies $z_1=z_3$.
However, we can easily remove this problem by reducing the redundant crossings in this case.
We expect that if $\mathcal{S}=\emptyset$ for a given link diagram, 
changing the diagram properly makes $\mathcal{S}\neq\emptyset$.

\begin{figure}[h]
\centering
  \includegraphics[scale=1]{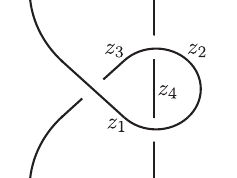}
  \caption{Diagram with $\mathcal{S}=\emptyset$}\label{nosolution}
\end{figure}

Note that the functions $\li(z)$ and $\log z$ are multi-valued functions. Therefore, to obtain well-defined values,
we have to select proper branch of the logarithm by choosing $\arg z$ and $\arg(1-z)$.
The following lemma shows why we consider the potential function $V_0$ instead of $V$.

\begin{lem}\label{branch}
Let $\bold{z}=(z_1,\ldots,z_n)\in\mathcal{S}$. 
For the potential function $V(z_1,\ldots,z_n)$, the value of $V_0(\bold{z})$ 
is invariant under a choice of branch of the logarithm modulo $4\pi^2$.
\end{lem}

\begin{proof}
Let $\li^*(z)$ and $\log^* z$ be the functions with different log-branch 
corresponding to an analytic continuation of $\li(z)$ and $\log z$ respectively.
Also let $V(z_1,\ldots,z_n)=\sum\pm\li(\frac{z_l}{z_m})$.
Then 
$$\li^*(\frac{z_l}{z_m})\equiv\li(\frac{z_l}{z_m})+2a\pi i\log\frac{z_l}{z_m}~~({\rm mod}~4\pi^2)$$
for a certain integer $a$, 
$$\log^*z_l\equiv\log z_l,~\log^*z_m\equiv\log z_m,~ \log^*\frac{z_l}{z_m}\equiv\log \frac{z_l}{z_m}~~({\rm mod}~2\pi i),$$
and, because of $\bold{z}=(z_1,\ldots,z_n)\in\mathcal{S}$, we have
$$z_l \frac{\partial ({\sum\pm}\li(\frac{z_l}{z_m}))}{\partial z_l}\equiv 
z_m \frac{\partial ({\sum\pm}\li(\frac{z_l}{z_m}))}{\partial z_m}\equiv 0~~({\rm mod}~2\pi i).$$

Therefore, 
\begin{eqnarray*}
\lefteqn{{\sum\pm}\left\{\li^*(\frac{z_l}{z_m})-\left(z_l \frac{\partial \li^*({z_l}/{z_m})}{\partial z_l}\right)\log^* z_l
-\left(z_m \frac{\partial \li^*({z_l}/{z_m})}{\partial z_m}\right)\log^* z_m\right\}}\\
&\equiv&{\sum\pm}\left\{\li(\frac{z_l}{z_m})+2a\pi i\log\frac{z_l}{z_m}-\left(z_l \frac{\partial \li({z_l}/{z_m})}{\partial z_l}\right)\log^* z_l-2a\pi i\log^* z_l\right.\\
&&\left.-\left(z_m \frac{\partial \li({z_l}/{z_m})}{\partial z_m}\right)\log^* z_m+2a\pi i\log^* z_m\right\}\\
&\equiv&{\sum\pm}\left\{\li(\frac{z_l}{z_m})+2a\pi i\log\frac{z_l}{z_m}-\left(z_l \frac{\partial \li({z_l}/{z_m})}{\partial z_l}\right)\log z_l-2a\pi i\log z_l\right.\\
&&\left.-\left(z_m \frac{\partial \li({z_l}/{z_m})}{\partial z_m}\right)\log z_m+2a\pi i\log z_m\right\}\\
&\equiv&{\sum\pm}\left\{\li(\frac{z_l}{z_m})-\left(z_l \frac{\partial \li({z_l}/{z_m})}{\partial z_l}\right)\log z_l
-\left(z_m \frac{\partial \li({z_l}/{z_m})}{\partial z_m}\right)\log z_m\right\}\modulos.
\end{eqnarray*}
The potential function $V_0$ is the summation of the above terms, so the proof follows. 

\end{proof}

\begin{lem}\label{lem1}
Let $\mathcal{S}=\cup_{j\in J}\mathcal{S}_j\subset\mathbb{C}^n$ be the solution set of $\mathcal{H}$ with $\mathcal{S}_j$ being a path component.
Assume $\mathcal{S}\neq\emptyset$.
Then, for any ${\bold z}=(z_1,\ldots,z_n)\in \mathcal{S}_j$,
$$V_0({\bold z})\equiv C_j\modulos,$$
where $C_j$ is a complex constant depending only on $j\in J$.
\end{lem} 

\begin{proof}

Note that $z_k\frac{\partial V}{\partial z_k}$ is continuous on $\mathcal{S}_j$ and $\exp(z_k\frac{\partial V}{\partial z_k})=1$ for
any ${\bold z}\in \mathcal{S}$. Therefore,
\begin{equation}\label{def_r}
z_k\frac{\partial V}{\partial z_k}=r_{j,k}\pi i,
\end{equation}
on $\mathcal{S}_j$ for an integer constant $r_{j,k}$ depending on $j$ and $k$. 
(The integer $r_{j,k}$ can be changed when $\bold z$ passes through the branch cut of the logarithm. In this case,
we change the branch cut so that $r_{j,k}$ is locally constant. The global invariance of $V_0$ is obtained by
the local invariance discussed below and Lemma \ref{branch}.)

For any path $$\mathbf{a}(t)=(\alpha_1(t),\ldots,\alpha_n(t)):[0,1]\longrightarrow \mathcal{S}_j,$$
using (\ref{def_r}) and the Chain rule, we have
\begin{eqnarray*}
\frac{d V_0}{dt}(\mathbf{a}(t))&=&\frac{d V}{dt}(\mathbf{a}(t))-\frac{d}{dt}\left(\sum_{k=1}^n r_{j,k}\pi i\log\alpha_k(t)\right)\\
&=&\sum_{k=1}^n\frac{\partial V}{\partial z_k}(\mathbf{a}(t))\alpha_k'(t)-\sum_{k=1}^n r_{j,k}\pi i\frac{\alpha_k'(t)}{\alpha_k(t)}\\
&=&\sum_{k=1}^n \frac{r_{j,k}\pi i}{\alpha_k(t)}\alpha_k'(t)-\sum_{k=1}^n r_{j,k}\pi i\frac{\alpha_k'(t)}{\alpha_k(t)}=0.
\end{eqnarray*}
This implies $V_0$ is constant on $\mathcal{S}_j$.

\end{proof}

Although we are considering the solution set $\mathcal{S}$ in $\mathbb{C}^n$, it is more natural to consider $\mathcal{S}$ as a subset
of the complex projective space $\mathbb{CP}^{n-1}$. This fact is not used in this article, but we show the following lemma
for reference.

\begin{cor}\label{lem23}
If ${\bold z}=(z_1,\ldots,z_n) \in \mathcal{S}_j$, then $\lambda{\bold z}:=(\lambda z_1,\ldots,\lambda z_n) \in \mathcal{S}_j$
for any nonzero complex number $\lambda$. Furthermore,
$$V_0({\bold z})\equiv V_0(\lambda{\bold z})\modulos.$$
\end{cor}

\begin{proof} The equations in $\mathcal{H}$ are products of the following terms
$$\exp\left(z_l\frac{\partial\li(z_l/z_m)}{\partial z_l}\right)=\left(1-\frac{z_l}{z_m}\right)^{-1} \text{ and }
\exp\left(z_m\frac{\partial\li(z_l/z_m)}{\partial z_m}\right)=\left(1-\frac{z_l}{z_m}\right),$$
which are represented only with ratios of the variables. This proves the first statement.

The second statement comes from Lemma \ref{lem1} by choosing a path from ${\bold z}$ to $\lambda{\bold z}$.

\end{proof}

\section{Octahedral triangulation of $\mathbb{S}^3\backslash (L\cup\{\text{two points}\})$}\label{sec3}

In this section, we describe an ideal triangulation of $\mathbb{S}^3\backslash (L\cup\{\text{two points}\})$.
We remark that this triangulation was already appeared in many different places 
because it naturally came from the link diagram. 
(For example, see Section 3 of \cite{Weeks05}.)
It was also appeared in Section 2.1 of \cite{Cho13b} and we named it (uncollapsed) Yokota triangulation.

To obtain the triangulation, we place an octahedron ${\rm A}_k{\rm B}_k{\rm C}_k{\rm D}_k{\rm E}_k{\rm F}_k$ 
on each crossing $k$ as in Figure \ref{twistocta} and twist it by identifying edges ${\rm B}_k{\rm F}_k$ to ${\rm D}_k{\rm F}_k$ and
${\rm A}_k{\rm E}_k$ to ${\rm C}_k{\rm E}_k$ respectively. The edges ${\rm A}_k{\rm B}_k$, ${\rm B}_k{\rm C}_k$,
${\rm C}_k{\rm D}_k$ and ${\rm D}_k{\rm A}_k$ are called {\it horizontal edges} and we sometimes express these edges
in the diagram as arcs around the crossing in the left hand side of Figure \ref{twistocta}.

\begin{figure}[h]
	\centering
\includegraphics[scale=0.9]{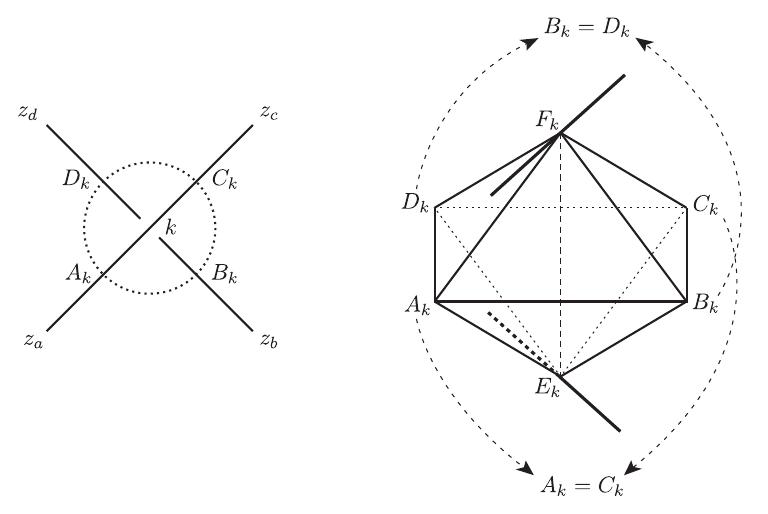}
 \caption{Octahedron on the crossing $k$}\label{twistocta}
\end{figure}

Then we glue faces of the octahedra following the sides of the diagram. 
Specifically, there are three gluing patterns as in Figure \ref{glue pattern}.
In each cases (a), (b) and (c), we identify the faces
$\triangle{\rm A}_{k}{\rm B}_{k}{\rm E}_{k}\cup\triangle{\rm C}_{k}{\rm B}_{k}{\rm E}_{k}$ to
$\triangle{\rm C}_{k+1}{\rm D}_{k+1}{\rm F}_{k+1}\cup\triangle{\rm C}_{k+1}{\rm B}_{k+1}{\rm F}_{k+1}$,
$\triangle{\rm B}_{k}{\rm C}_{k}{\rm F}_{k}\cup\triangle{\rm D}_{k}{\rm C}_{k}{\rm F}_{k}$ to
$\triangle{\rm D}_{k+1}{\rm C}_{k+1}{\rm F}_{k+1}\cup\triangle{\rm B}_{k+1}{\rm C}_{k+1}{\rm F}_{k+1}$
and
$\triangle{\rm A}_{k}{\rm B}_{k}{\rm E}_{k}\cup\triangle{\rm C}_{k}{\rm B}_{k}{\rm E}_{k}$ to
$\triangle{\rm C}_{k+1}{\rm B}_{k+1}{\rm E}_{k+1}\cup\triangle{\rm A}_{k+1}{\rm B}_{k+1}{\rm E}_{k+1}$
respectively.

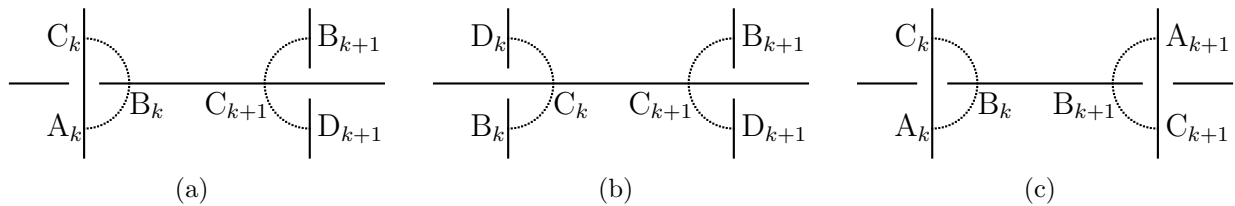
\begin{figure}[h]
\centering
  \subfigure[]
  {\begin{picture}(5,2)\thicklines
   \put(1,1){\arc[5](0,-0.6){180}}
   \put(4,1){\arc[5](0,0.6){180}}
   \put(1.2,1){\line(1,0){3.8}}
   \put(1,2){\line(0,-1){2}}
   \put(4,0){\line(0,1){0.8}}
   \put(4,2){\line(0,-1){0.8}}
   \put(0.8,1){\line(-1,0){0.8}}
   \put(0.5,0.3){${\rm A}_k$}
   \put(1.6,0.6){${\rm B}_k$}
   \put(0.5,1.5){${\rm C}_k$}
   \put(4.1,0.3){${\rm D}_{k+1}$}
   \put(2.6,0.6){${\rm C}_{k+1}$}
   \put(4.1,1.5){${\rm B}_{k+1}$}
%   \put(2.5,1.2){$z_a$}
  \end{picture}}\hspace{0.5cm}
  \subfigure[]
  {\begin{picture}(5,2)\thicklines
   \put(1,1){\arc[5](0,-0.6){180}}
   \put(4,1){\arc[5](0,0.6){180}}
   \put(5,1){\line(-1,0){5}}
   \put(1,2){\line(0,-1){0.8}}
   \put(1,0){\line(0,1){0.8}}
   \put(4,2){\line(0,-1){0.8}}
   \put(4,0){\line(0,1){0.8}}
   \put(0.5,0.3){${\rm B}_k$}
   \put(1.6,0.6){${\rm C}_k$}
   \put(0.5,1.5){${\rm D}_k$}
   \put(4.1,0.3){${\rm D}_{k+1}$}
   \put(2.6,0.6){${\rm C}_{k+1}$}
   \put(4.1,1.5){${\rm B}_{k+1}$}
%   \put(2.5,1.2){$z_b$}
  \end{picture}}\hspace{0.5cm}
  \subfigure[]
  {\begin{picture}(5,2)\thicklines
   \put(1,1){\arc[5](0,-0.6){180}}
   \put(4,1){\arc[5](0,0.6){180}}
   \put(4.2,1){\line(1,0){0.8}}
   \put(1,2){\line(0,-1){2}}
   \put(0.8,1){\line(-1,0){0.8}}
   \put(4,2){\line(0,-1){2}}
   \put(1.2,1){\line(1,0){2.6}}
   \put(0.5,0.3){${\rm A}_k$}
   \put(1.6,0.6){${\rm B}_k$}
   \put(0.5,1.5){${\rm C}_k$}
   \put(4.1,0.3){${\rm C}_{k+1}$}
   \put(2.6,0.6){${\rm B}_{k+1}$}
   \put(4.1,1.5){${\rm A}_{k+1}$}
%   \put(2.5,1.2){$z_c$}
  \end{picture}}
  \caption{Three gluing patterns}\label{glue pattern}
\end{figure} 

Note that this gluing process identifies vertices $\{{\rm A}_k, {\rm C}_k\}$ to one point, denoted by $-\infty$,
and $\{{\rm B}_k, {\rm D}_k\}$ to another point, denoted by $\infty$, and finally $\{{\rm E}_k, {\rm F}_k\}$ to
the other points, denoted by ${\rm P}_j$ where $j=1,\ldots,s$ and $s$ is the number of the components of the link $L$. 
The regular neighborhoods of $-\infty$ and $\infty$ are 3-balls and that of $\cup_{j=1}^s P_j$ is
a tubular neighborhood of the link $L$. 
Therefore, if we remove the vertices ${\rm P}_1,\ldots,{\rm P}_s$ from the gluing, 
then we obtain a triangulation of $\mathbb{S}^3\backslash L$, denoted by $\mathcal{T}$.
On the other hand, if we remove all the vertices of the gluing, 
the result becomes an ideal triangulation of $\mathbb{S}^3\backslash (L\cup\{\pm\infty\})$.
We call this ideal triangulation {\it octahedral triangulation} and denote it by $\mathcal{T}'$.

Let $M=\mathbb{S}^3\backslash L$ and $M'=\mathbb{S}^3\backslash (L\cup\{\pm\infty\})$.
Then there exists a continuous deformation of the developing maps from 
$\widetilde{M}\longrightarrow\mathbb{H}^3$ to 
$\widetilde{M'}\longrightarrow\mathbb{H}^3$,
called Thurston's spinning construction. Section 3 of \cite{Tillmann13} explains this construction 
for closed manifolds, but it can be applied to our triangulation $\mathcal{T}$ by fixing ideal points ${\rm P}_1,\ldots,{\rm P}_s$
and sending points $\pm\infty$ to $\partial\overline{\mathbb{H}^3}=\mathbb{CP}^1$.
Therefore, the parameter space of $\mathcal{T}'$ in \cite{Tillmann13} determines the complex volume of $M$.
(We will apply Zickert's formula of \cite{Zickert09} to $\mathcal{T}'$ for calculating the complex volumes of $M$. 
See Section \ref{sec4} for details.)

To describe the parameter space of the octahedral triangulation $\mathcal{T}'$, we divide each ideal octahedron 
${\rm A}_k{\rm B}_k{\rm C}_k{\rm D}_k{\rm E}_k{\rm F}_k$ into four ideal tetrahedra
${\rm A}_k{\rm B}_k{\rm E}_k{\rm F}_k$, ${\rm B}_k{\rm C}_k{\rm E}_k{\rm F}_k$,
${\rm C}_k{\rm D}_k{\rm E}_k{\rm F}_k$ and ${\rm D}_k{\rm A}_k{\rm E}_k{\rm F}_k$.
When $z_a$, $z_b$, $z_c$ and $z_d$ are assigned 
to the sides around the octahedron as in Figure \ref{parameterizing},
we parametrize each tetrahedra by assigning {\it shape parameters} 
$\frac{z_b}{z_a}$, $\frac{z_c}{z_b}$, $\frac{z_d}{z_c}$ and $\frac{z_a}{z_d}$ to the horizontal edges 
${\rm A}_k{\rm B}_k$, ${\rm B}_k{\rm C}_k$, ${\rm C}_k{\rm D}_k$ and ${\rm D}_k{\rm A}_k$ respectively.

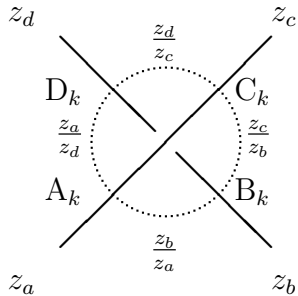
\begin{figure}[h]\centering
\begin{picture}(6,4)  
  \setlength{\unitlength}{0.7cm}\thicklines
        \put(4,3){\arc[5](1,1){360}}
    \put(6,5){\line(-1,-1){4}}
    \put(2,5){\line(1,-1){1.8}}
    \put(4.2,2.8){\line(1,-1){1.8}}
    \put(3.7,0.8){$\frac{z_b}{z_a}$}
    \put(5.5,3){$\frac{z_c}{z_b}$}
    \put(3.7,4.8){$\frac{z_d}{z_c}$}
    \put(1.9,3){$\frac{z_a}{z_d}$}
    \put(1,5.3){$z_d$}
    \put(6,5.3){$z_c$}
    \put(1,0.2){$z_a$}
    \put(6,0.2){$z_b$}
    \put(1.7,1.9){${\rm A}_k$}
    \put(5.3,1.9){${\rm B}_k$}
    \put(5.3,3.8){${\rm C}_k$}
    \put(1.7,3.8){${\rm D}_k$}
  \end{picture}\caption{Parametrizing tetrahedra}\label{parameterizing}
  \end{figure}

Note that if we assign a shape parameter $u\in\mathbb{C}\backslash\{0,1\}$ to an edge of an ideal tetrahedron,
then the other edges are also parametrized by $u, u':=\frac{1}{1-u}$ and $u'':=1-\frac{1}{u}$
as in Figure \ref{pic10}.

\begin{figure}[h]
\begin{center}
  {\setlength{\unitlength}{0.4cm}
  \begin{picture}(12,10)\thicklines
   \put(1,1){\line(1,0){8}}
   \put(1,1){\line(1,2){4}}
   \put(5,9){\line(1,-2){4}}
   \put(9,1){\line(1,2){2}}
   \put(5,9){\line(3,-2){6}}
   \dashline{0.5}(1,1)(11,5)
%   \put(0,0){A}
%   \put(9,0){B}
%   \put(11,5){C}
%   \put(4.5,9.5){D}
   \put(5,0){$u$}
   \put(8,7){$u$}
   \put(2.2,5){$u'$}
   \put(10.5,3){$u'$}
   \put(5,3){$u''$}
   \put(7,5){$u''$}
  \end{picture}}
  \caption{Parametrization of an ideal tetrahedron with a parameter $u$}\label{pic10}
\end{center}
\end{figure}
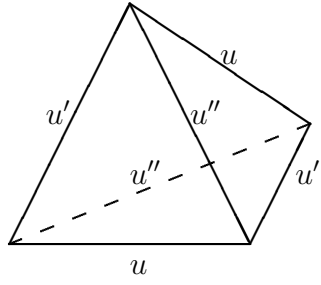

For a given ideal triangulation of $\mathbb{S}^3\backslash L$
or $\mathbb{S}^3\backslash (L\cup\{\pm\infty\})$,
we require two conditions to obtain the complete hyperbolic structure;
the product of shape parameters on an edge is one for all edges, and the holonomies
of meridian and longitude act as translations on the cusp.
The former are called {\it Thurston's gluing equations} and the latter {\it completeness conditions}.
Note that these conditions are expressed as equations of shape parameters.
The whole set of these equations are called {\it the hyperbolicity equations}.
The works of Luo, Tillmann and others in \cite{Tillmann13} and \cite{Tillmann11} use only Thurston's gluing equations, but, in this article,
we also require completeness conditions. Therefore, if $\bold{z}$ is a solution of the hyperbolicity equations,
then the induced representation $\rho_{\bold{z}}:\pi_1(\mathbb{S}^3\backslash L)\longrightarrow{\rm PSL}(2,\mathbb{C})$
is boundary-parabolic.

The rest of this section is devoted to the proof of Proposition \ref{pro1}. 
Note that Proposition \ref{pro1} was already appeared and proved in \cite{Yokota10} in a slightly different way.

\begin{proof}[Proof of Proposition \ref{pro1}]
For each octahedron in Figure \ref{twistocta} of the octahedral triangulation, let $\mathcal{A}$ be the set of horizontal edges
${\rm A}_k{\rm B}_k$, ${\rm B}_k{\rm C}_k$, ${\rm C}_k{\rm D}_k$ and ${\rm D}_k{\rm A}_k$ of all crossings $k$. 
Let $\mathcal{B}$ be the set of edges ${\rm B}_k{\rm F}_k$, ${\rm D}_k{\rm F}_k$,
${\rm A}_k{\rm E}_k$, ${\rm C}_k{\rm E}_k$ of all crossings and other edges glued to them. Let $\mathcal{C}$ be
the set of edges ${\rm E}_k{\rm F}_k$ of all crossings and let $\mathcal{D}$ be set of the other edges in the triangulation. Note that
if the diagram $D$ is alternating, then $\mathcal{D}=\emptyset$.

The rule of assigning shape parameters to horizontal edges 
makes the edge conditions of $\mathcal{A}$ and $\mathcal{C}$ hold trivially.

\begin{lem}\label{lem31}
The set of equations $\mathcal{H}$ consists of the completeness conditions along the meridian and 
Thurston's gluing equations of the elements in $\mathcal{D}$.
\end{lem}

\begin{proof}
Consider the following three cases in Figure \ref{cases}.
We call the case (a) {\it alternating gluing} and the other cases (b) and (c) {\it non-alternating gluings}.
Note that elements of $\mathcal{D}$ appear only in non-alternating gluings. 
(Specifically ${\rm C}_k{\rm F}_k={\rm C}_{k+1}{\rm F}_{k+1}\in\mathcal{D}$ in the case (b) and
${\rm B}_k{\rm E}_k={\rm B}_{k+1}{\rm E}_{k+1}\in\mathcal{D}$ in the case (c).)

\begin{figure}[h]
\centering
  \subfigure[]
  {\begin{picture}(5,2)\thicklines
   \put(1,1){\arc[5](0,-0.6){180}}
   \put(4,1){\arc[5](0,0.6){180}}
   \put(1.2,1){\line(1,0){3.8}}
   \put(1,2){\line(0,-1){2}}
   \put(4,0){\line(0,1){0.8}}
   \put(4,2){\line(0,-1){0.8}}
   \put(0.8,1){\line(-1,0){0.8}}
   \put(0.5,0.3){${\rm A}_k$}
   \put(1.6,0.6){${\rm B}_k$}
   \put(0.5,1.5){${\rm C}_k$}
   \put(4.1,0.3){${\rm D}_{k+1}$}
   \put(2.6,0.6){${\rm C}_{k+1}$}
   \put(4.1,1.5){${\rm B}_{k+1}$}
   \put(2.5,1.2){$z_l$}
   \put(0.6,0){$z_a$}
   \put(4.1,0){$z_b$}
   \put(4.1,2){$z_c$}
   \put(0.6,2){$z_d$}
  \end{picture}}\hspace{0.5cm}
  \subfigure[]
  {\begin{picture}(5,2)\thicklines
   \put(1,1){\arc[5](0,-0.6){180}}
   \put(4,1){\arc[5](0,0.6){180}}
   \put(5,1){\line(-1,0){5}}
   \put(1,2){\line(0,-1){0.8}}
   \put(1,0){\line(0,1){0.8}}
   \put(4,2){\line(0,-1){0.8}}
   \put(4,0){\line(0,1){0.8}}
   \put(0.5,0.3){${\rm B}_k$}
   \put(1.6,0.6){${\rm C}_k$}
   \put(0.5,1.5){${\rm D}_k$}
   \put(4.1,0.3){${\rm D}_{k+1}$}
   \put(2.6,0.6){${\rm C}_{k+1}$}
   \put(4.1,1.5){${\rm B}_{k+1}$}
   \put(2.5,1.2){$z_l$}
   \put(0.6,0){$z_a$}
   \put(4.1,0){$z_b$}
   \put(4.1,2){$z_c$}
   \put(0.6,2){$z_d$}
  \end{picture}}\hspace{0.5cm}
  \subfigure[]
  {\begin{picture}(5,2)\thicklines
   \put(1,1){\arc[5](0,-0.6){180}}
   \put(4,1){\arc[5](0,0.6){180}}
   \put(4.2,1){\line(1,0){0.8}}
   \put(1,2){\line(0,-1){2}}
   \put(0.8,1){\line(-1,0){0.8}}
   \put(4,2){\line(0,-1){2}}
   \put(1.2,1){\line(1,0){2.6}}
   \put(0.5,0.3){${\rm A}_k$}
   \put(1.6,0.6){${\rm B}_k$}
   \put(0.5,1.5){${\rm C}_k$}
   \put(4.1,0.3){${\rm C}_{k+1}$}
   \put(2.6,0.6){${\rm B}_{k+1}$}
   \put(4.1,1.5){${\rm A}_{k+1}$}
   \put(2.5,1.2){$z_l$}
   \put(0.6,0){$z_a$}
   \put(4.1,0){$z_b$}
   \put(4.1,2){$z_c$}
   \put(0.6,2){$z_d$}
  \end{picture}}
  \caption{Three cases of gluings}\label{cases}
\end{figure}
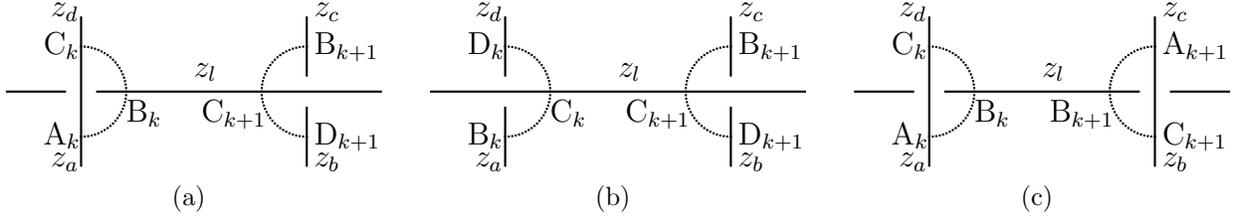 

The variables $z_a, z_b, z_c, z_d$ and $z_l$ are assigned to each sides in Figure \ref{cases}. 
The potential function $V^{(a)}$ of the four corners in Figure \ref{cases}(a) is defined by
$$V^{(a)}=\li\left(\frac{z_l}{z_a}\right)
+\li\left(\frac{z_b}{z_l}\right)-\li\left(\frac{z_c}{z_l}\right)-\li\left(\frac{z_l}{z_d}\right),$$
and it induces the following equation
\begin{eqnarray}
\exp\left(z_l\frac{\partial V}{\partial z_l}\right)&=&\exp\left(z_l\frac{\partial V^{(a)}}{\partial z_l}\right)\nonumber\\
&=&\left(1-\frac{z_l}{z_a}\right)^{-1}\left(1-\frac{z_b}{z_l}\right)
\left(1-\frac{z_c}{z_l}\right)^{-1}\left(1-\frac{z_l}{z_d}\right)=1\in\mathcal{H}.\label{a}
\end{eqnarray}

On the other hand, the cusp along the side $z_l$ in Figure \ref{cases}(a)
can be visualized by the annulus in Figure \ref{Figure a}.
In Figure \ref{Figure a}, $a_k$, $b_k$, $c_k$, $b_{k+1}$, $c_{k+1}$, $d_{k+1}$ are the points of the cusp, which lie on
the edges ${\rm A}_k{\rm E}_k$, ${\rm B}_k{\rm E}_k$, ${\rm C}_k{\rm E}_k$, 
${\rm B}_{k+1}{\rm F}_{k+1}$, ${\rm C}_{k+1}{\rm F}_{k+1}$, ${\rm D}_{k+1}{\rm F}_{k+1}$ respectively,
and $m$ is the meridian of the cusp. 

\begin{figure}[h]
\centering
\includegraphics[scale=0.85]{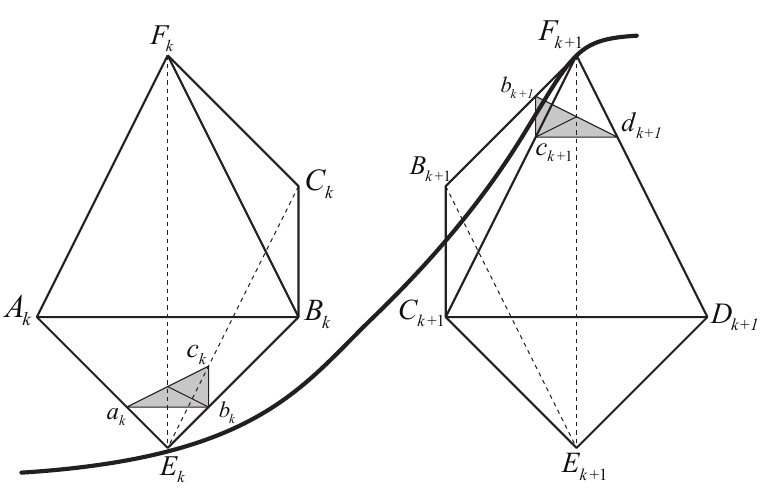}
\includegraphics[scale=0.85]{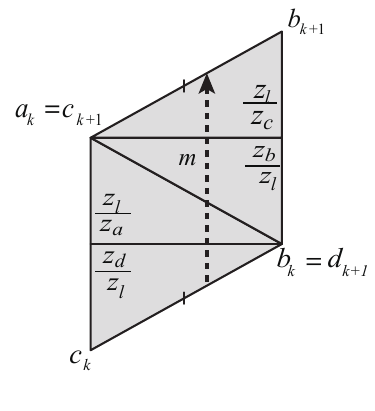}
 \caption{Cusp diagram of Figure \ref{cases}(a)}\label{Figure a}
\end{figure}

The completeness condition along $m$ in Figure \ref{Figure a} becomes
$$\left\{\left(\frac{z_d}{z_l}\right)''\left(\frac{z_l}{z_a}\right)'\right\}^{-1}\left(\frac{z_b}{z_l}\right)'\left(\frac{z_l}{z_c}\right)''
=\left(1-\frac{z_l}{z_d}\right)^{-1}\left(1-\frac{z_l}{z_a}\right)\left(1-\frac{z_b}{z_l}\right)^{-1}
\left(1-\frac{z_c}{z_l}\right)=1,
$$
which is equivalent to (\ref{a}).

The potential function $V^{(b)}$ of the four corners in Figure \ref{cases}(b) is defined by
$$V^{(b)}=-\li\left(\frac{z_a}{z_l}\right)
+\li\left(\frac{z_b}{z_l}\right)-\li\left(\frac{z_c}{z_l}\right)+\li\left(\frac{z_d}{z_l}\right),$$
and it induces the following equation
\begin{eqnarray}
\exp\left(z_l\frac{\partial V}{\partial z_l}\right)&=&\exp\left(z_l\frac{\partial V^{(b)}}{\partial z_l}\right)\nonumber\\
&=&\left(1-\frac{z_a}{z_l}\right)^{-1}\left(1-\frac{z_b}{z_l}\right)
\left(1-\frac{z_c}{z_l}\right)^{-1}\left(1-\frac{z_d}{z_l}\right)=1\in\mathcal{H}.\label{b}
\end{eqnarray}

On the other hand, the cusp along the side $z_l$ in Figure \ref{cases}(b)
can be visualized by Figure \ref{Figure b}.
In Figure \ref{Figure b}, $b_k$, $c_k$, $d_k$, $b_{k+1}$, $c_{k+1}$, $d_{k+1}$ are the points of the cusp, which lie on
the edges ${\rm B}_k{\rm F}_k$, ${\rm C}_k{\rm F}_k$, ${\rm D}_k{\rm F}_k$,
${\rm B}_{k+1}{\rm F}_{k+1}$, ${\rm C}_{k+1}{\rm F}_{k+1}$, ${\rm D}_{k+1}{\rm F}_{k+1}$ respectively,
and the edges $c_k d_k$ and $c_k b_k$ are identified to $c_{k+1} b_{k+1}$ and $c_{k+1} d_{k+1}$ respectively.

\begin{figure}[h]
\centering
\includegraphics[scale=0.82]{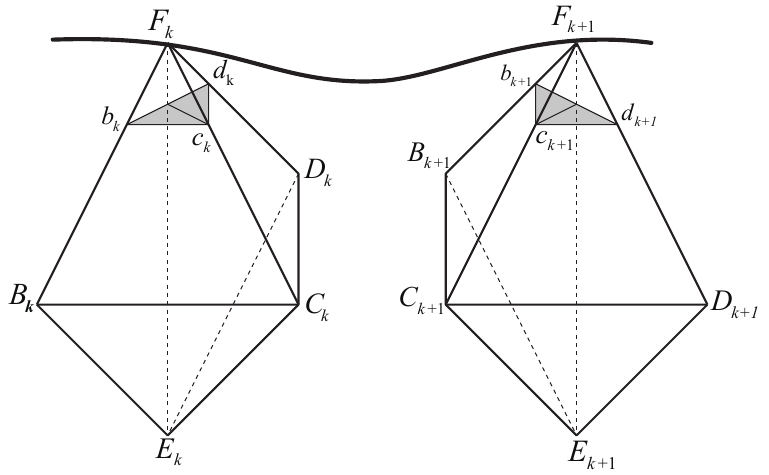}
\includegraphics[scale=0.9]{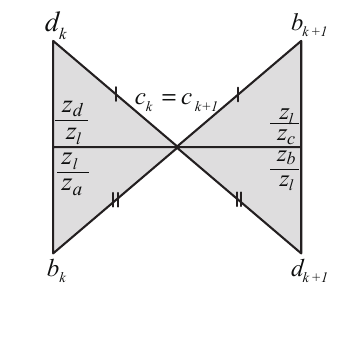}
 \caption{Cusp diagram of Figure \ref{cases}(b)}\label{Figure b}
\end{figure}

Thurston's gluing equation of the edge ${\rm C}_k{\rm F}_k={\rm C}_{k+1}{\rm F}_{k+1}\in\mathcal{D}$ (around $c_k=c_{k+1}$)
in Figure \ref{Figure b} becomes
$$\left(\frac{z_l}{z_a}\right)''\left(\frac{z_b}{z_l}\right)'\left(\frac{z_l}{z_c}\right)''\left(\frac{z_d}{z_l}\right)'
=\left(1-\frac{z_a}{z_l}\right)\left(1-\frac{z_b}{z_l}\right)^{-1}
\left(1-\frac{z_c}{z_l}\right)\left(1-\frac{z_d}{z_l}\right)^{-1}=1,$$
which is equivalent to (\ref{b}).

The potential function $V^{(c)}$ of the four corners in Figure \ref{cases}(c) is defined by
$$V^{(c)}=\li\left(\frac{z_l}{z_a}\right)
-\li\left(\frac{z_l}{z_b}\right)+\li\left(\frac{z_l}{z_c}\right)-\li\left(\frac{z_l}{z_d}\right),$$
and it induces the following equation
\begin{eqnarray}
\exp\left(z_l\frac{\partial V}{\partial z_l}\right)&=&
\exp\left(z_l\frac{\partial V^{(c)}}{\partial z_l}\right)\nonumber\\
&=&\left(1-\frac{z_l}{z_a}\right)^{-1}\left(1-\frac{z_l}{z_b}\right)
\left(1-\frac{z_l}{z_c}\right)^{-1}\left(1-\frac{z_l}{z_d}\right)=1\in\mathcal{H}.\label{c}
\end{eqnarray}

On the other hand, the cusp along the side $z_l$ in Figure \ref{cases}(c)
can be visualized by Figure \ref{Figure c}.
In Figure \ref{Figure c}, $a_k$, $b_k$, $c_k$, $a_{k+1}$, $b_{k+1}$, $c_{k+1}$ are the points of the cusp, which lie on
the edges ${\rm A}_k{\rm E}_k$, ${\rm B}_k{\rm E}_k$, ${\rm C}_k{\rm E}_k$,
${\rm A}_{k+1}{\rm E}_{k+1}$, ${\rm B}_{k+1}{\rm E}_{k+1}$, ${\rm C}_{k+1}{\rm E}_{k+1}$ respectively,
and the edges $b_k a_k$ and $b_k c_k$ are identified to $b_{k+1} c_{k+1}$ and $b_{k+1} a_{k+1}$ respectively.

\begin{figure}[h]
\includegraphics[scale=0.8]{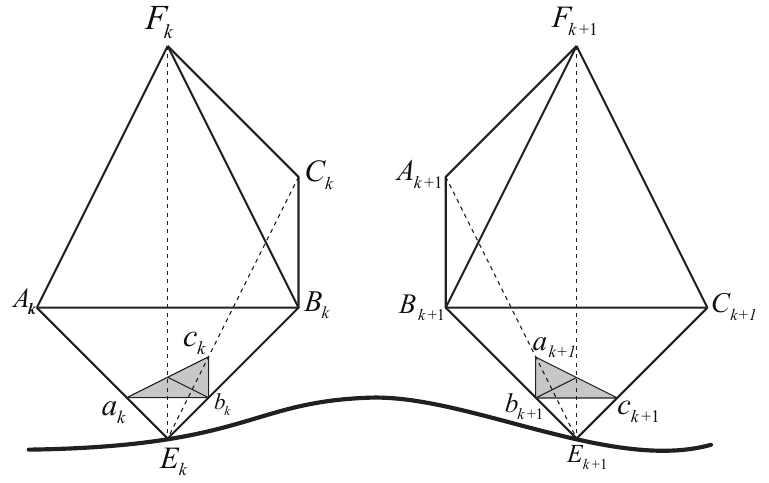}
\includegraphics[scale=0.9]{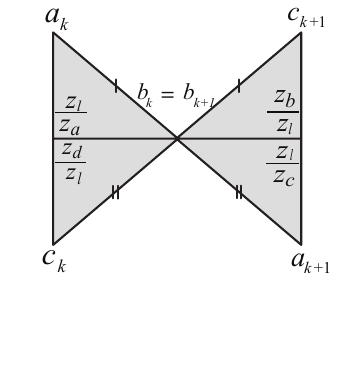}
 \caption{Cusp diagram of Figure \ref{cases}(c)}\label{Figure c}
\end{figure}

Thurston's gluing equation of the edge ${\rm B}_k{\rm E}_k={\rm B}_{k+1}{\rm E}_{k+1}\in\mathcal{D}$ (around $b_k=b_{k+1}$)
in Figure \ref{Figure c} becomes
$$\left(\frac{z_l}{z_a}\right)'\left(\frac{z_b}{z_l}\right)''\left(\frac{z_l}{z_c}\right)'\left(\frac{z_d}{z_l}\right)''
=\left(1-\frac{z_l}{z_a}\right)^{-1}\left(1-\frac{z_l}{z_b}\right)
\left(1-\frac{z_l}{z_c}\right)^{-1}\left(1-\frac{z_l}{z_d}\right)=1,$$
which is equivalent to (\ref{c}). It completes the proof of Lemma \ref{lem31}.

\end{proof}

We remark that the cusp diagram of alternating gluing becomes an annulus, but that of non-alternating gluing
eventually becomes a part of annulus. 
This comes from the cusp diagrams of the two cases in Figure \ref{nonalter1} and Figure \ref{nonalter2}.
\begin{figure}[h]
\centering
\subfigure[]{\includegraphics[scale=0.8]{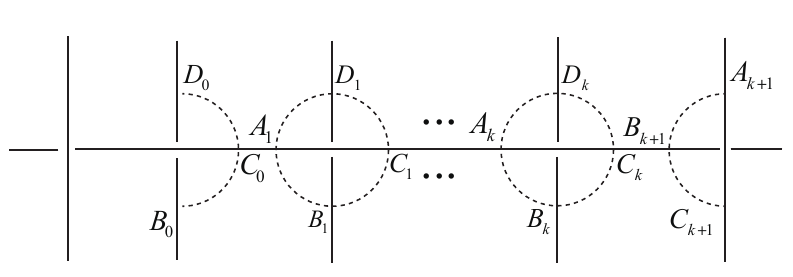}}
\subfigure[]{\includegraphics[scale=1]{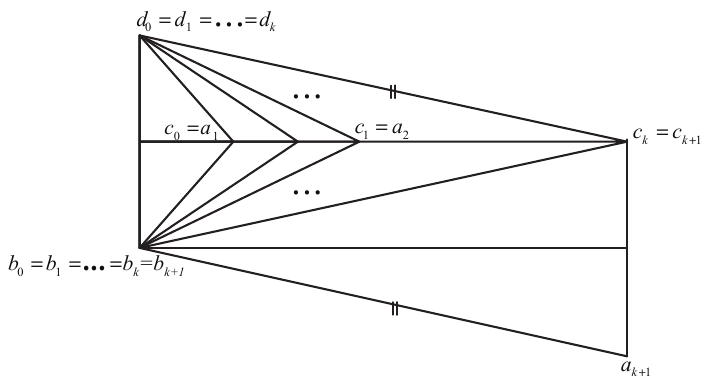}}
 \caption{First non-alternating gluing and its cusp diagram ($k\geq 2$)}\label{nonalter1}
\end{figure}

\begin{figure}[h]
\centering
\subfigure[]{\includegraphics[scale=0.8]{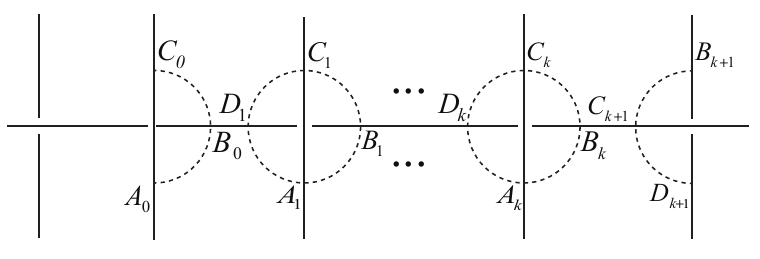}}
\subfigure[]{\includegraphics[scale=1]{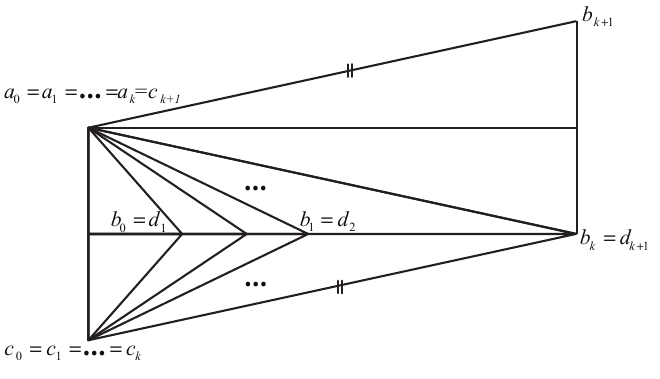}}
 \caption{Second non-alternating gluing and its cusp diagram ($k\geq 2$)}\label{nonalter2}
\end{figure}

Due to the completeness conditions in $\mathcal{H}$, the edges $d_k c_k$ and $b_k c_k$ are identified 
to $b_{k+1} a_{k+1}$ and $b_{k+1} c_{k+1}$ respectively in Figure \ref{nonalter1}(b),
and the edges $a_k b_k$ and $c_k b_k$ are identified to $c_{k+1} d_{k+1}$ and $c_{k+1}b_{k+1}$ respectively in Figure \ref{nonalter2}(b).
These identifications make the cusp diagrams topological annuli.
Furthermore, due to the Thurston's gluing equations of the edges in $\mathcal{C}\cup\mathcal{D}$,
the annuli have Euclidean structures.

To complete the proof of Proposition \ref{pro1}, we show the completeness conditions in $\mathcal{H}$ and 
Thurston's gluing equations of the the edges in 
$\mathcal{A}\cup\mathcal{C}\cup\mathcal{D}$ induce the other gluing equations of the edges in $\mathcal{B}$.

Consider the crossing $k$ in Figure \ref{edgeB1}.  
The crossing $l$ is the previous under-crossing of $k$ and the crossing $m$ is the next under-crossing.

\begin{figure}
\centering 
\includegraphics[scale=0.95]{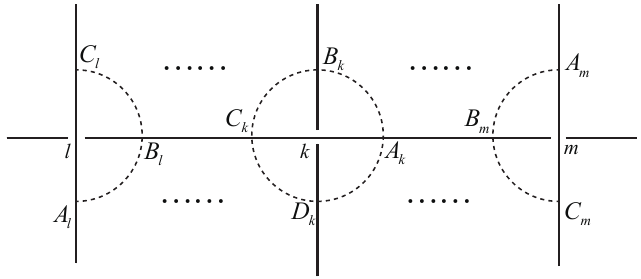}
\caption{The case of ${\rm B}_k{\rm F}_k={\rm D}_k{\rm F}_k\in\mathcal{B}$}\label{edgeB1}
\end{figure}

Thurston's gluing equation of ${\rm B}_k{\rm F}_k={\rm D}_k{\rm F}_k\in\mathcal{B}$
follows from the gluing equation around $b_k=d_k$ of the cusp diagram in Figure \ref{edgeB2}
since $c_l d_k$ and $d_k a_m$ are parallel to $a_l b_k$ and $b_k c_m$ respectively. 
%which is an Euclidean annulus.

\begin{figure}
\centering 
\includegraphics[scale=0.72]{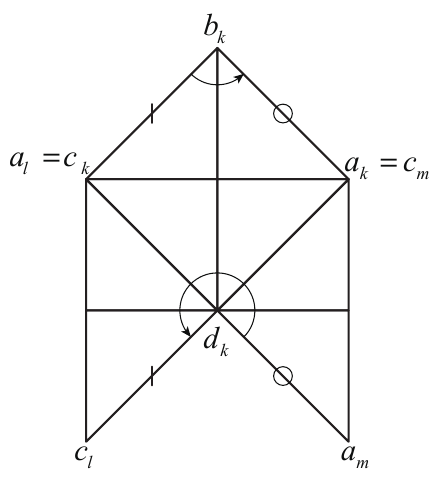}
\caption{The cusp diagram of Figure \ref{edgeB1}}\label{edgeB2}
\end{figure}

The proof of the case of ${\rm A}_k{\rm E}_k={\rm C}_k{\rm E}_k\in\mathcal{B}$ is also obtained by
considering Figure \ref{edgeB3} and following the same argument as before. 

\begin{figure}
\centering 
\includegraphics[scale=0.95]{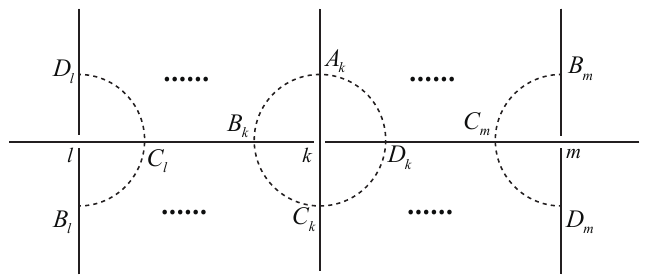}
\caption{The case of ${\rm A}_k{\rm E}_k={\rm C}_k{\rm E}_k\in\mathcal{B}$}\label{edgeB3}
\end{figure}

As a conclusion, we showed $\mathcal{H}$ induces Thurston's gluing equations of all the edges. 
The completeness conditions along the meridian in $\mathcal{H}$ and all the gluing equations together induce the completeness condition
along the longitude, so $\mathcal{H}$ induces the whole hyperbolicity equations.

\end{proof}

\section{Proof of Theorem \ref{thm1}}\label{sec4}

In this section, we always assume ${\bold z}=(z_1,\ldots,z_n)$ is a solution in $\mathcal{S}_j$ and drop the index $j$ of $r_{j,k}$ in (\ref{def_r}).

The main technique of the proof of Theorem \ref{thm1} is the extended Bloch group theory in \cite{Zickert09}.
To apply it, we first define the vertex ordering of the octahedral triangulation. In Figure \ref{twistocta}, we assign 0 and 1
to the vertices ${\rm E}_k$ and ${\rm F}_k$ respectively, 2 to the vertices ${\rm A}_k$ and ${\rm C}_k$, and
3 to the vertices ${\rm B}_k$ and ${\rm D}_k$. This assignment induces the vertex orderings of the four tetrahedra.

Note that the vertex ordering of each tetrahedron induces
the orientations of the edges and the tetrahedron. The induced orientation of the tetrahedron
can be different from the original orientation induced by the triangulation.
For example, the tetrahedra ${\rm E}_k{\rm F}_k{\rm C}_k{\rm B}_k$ and ${\rm E}_k{\rm F}_k{\rm A}_k{\rm D}_k$
in Figure \ref{twistocta}
are the cases. If the two orientations are the same, we define the sign of the tetrahedron $\sigma=1$, and if they are different,
then $\sigma=-1$.

One important property of this vertex orientation is that 
when two edges are glued together in the triangulation, 
the orientations of the two edges induced by each vertex orderings coincide.
(We call this condition {\it edge-orientation consistency}.)
Because of this property, we can apply the formula in \cite{Zickert09}.

The triangulation we are using is an ideal triangulation, so we already parametrized all ideal tetrahedra of the triangulation
by assigning shape parameters to horizontal edges in Section \ref{sec3}.
For each tetrahedron with the vertex-orientation, we define an element of the extended pre-Bloch group
$\sigma[u^{\sigma};p,q]\in\widehat{\mathcal{P}}(\mathbb{C})$, where $\sigma$ is the sign of the tetrahedron, 
$u$ is the shape parameter assigned to the edge connecting the $0$th and $1$st vertices, and $p,q$ are certain integers.

Zickert suggested a way to determine $p$ and $q$ from the developing map 
of the representation $\rho:\pi_1(M)\rightarrow{\rm PSL}(2,\mathbb{C})$ of a hyperbolic manifold $M$ in \cite{Zickert09},
and showed that
\begin{equation}\label{formula}
\widehat{L}(\sum\sigma[u^{\sigma};p,q])\equiv i(\vol(\rho)+i\,\cs(\rho))\modulo,
\end{equation}
where the summation is over all tetrahedra and 
$$\widehat{L}([u;p,q])=\li(u)-\frac{\pi^2}{6}+\frac{1}{2}q\pi i\log u+\frac{1}{2}\log(1-u)(\log u+p\pi i)$$
is a complex valued function defined on $\widehat{\mathcal{P}}(\mathbb{C})$.

Although our ideal triangulation $\mathcal{T}'$ is that of $\mathbb{S}^3\backslash (L\cup\{\pm\infty\})$, 
the formula of \cite{Zickert09} is still valid because of Thurston's spinning construction. 
Theorem 4.11 of \cite{Zickert09} already considered our case and 
the developing map of the representation is the one obtained by Thurston's spinning construction. 
(The map sends $\pm\infty$ to ideal points corresponding to the trivial ends.)

To determine $p, q$ of $\sigma[u^{\sigma};p,q]$ of a tetrahedron with vertex orientation, we assign certain complex numbers 
$g_{jk}$ to the edge connecting the $j$th and $k$th vertices, where $j,k\in\{0,1,2,3\}$ and $j<k$.
We assume $g_{jk}$ satisfies the property that 
if two edges are glued together in the triangulation, then the assigned $g_{jk}$'s of the edges coincide.
We do not use the exact values of $g_{jk}$ in this article, but remark that there is an explicit method in \cite{Zickert09} 
for calculating these numbers using the developing map. With the given numbers $g_{jk}$, we can calculate $p,q$ using the following equations,
which appeared as equation (3.5) in \cite{Zickert09}:
  \begin{eqnarray}
      \left\{\begin{array}{ll}
      p\pi i =-\log u^{\sigma}+\log g_{03}+\log g_{12}-\log g_{02}-\log g_{13},\\
      q\pi i =\log(1-u^{\sigma})+\log g_{02}+\log g_{13}-\log g_{23}-\log g_{01}.
      \end{array}\right.\label{pq}
  \end{eqnarray}

To avoid confusion, we use variables $\alpha_m, \beta_m, \gamma_m, \delta_m$ instead of $g_{jk}$.
We assign $\alpha_m$ and $\beta_m$ to non-horizontal edges as in Figure \ref{labeling}, where $m=a,b,c,d$. 
We also assign $\gamma_l$ to horizontal edges and $\delta_k$ to the edge ${\rm E}_k{\rm F}_k$ inside the octahedron.
Although we have $\alpha_a=\alpha_c$ and $\beta_b=\beta_d$, we use $\alpha_a$ for the tetrahedron
${\rm E}_k{\rm F}_k{\rm A}_k{\rm B}_k$ and ${\rm E}_k{\rm F}_k{\rm A}_k{\rm D}_k$, $\alpha_c$ for
${\rm E}_k{\rm F}_k{\rm C}_k{\rm B}_k$ and ${\rm E}_k{\rm F}_k{\rm C}_k{\rm D}_k$, $\beta_b$ for
${\rm E}_k{\rm F}_k{\rm A}_k{\rm B}_k$ and ${\rm E}_k{\rm F}_k{\rm C}_k{\rm B}_k$, $\beta_d$ for
${\rm E}_k{\rm F}_k{\rm C}_k{\rm D}_k$ and ${\rm E}_k{\rm F}_k{\rm A}_k{\rm D}_k$.
We assign vertex orderings of the tetrahedra in Figure \ref{labeling} by assigning 0 to ${\rm E}_k$,
1 to ${\rm F}_k$, 2 to ${\rm A}_k$ and ${\rm C}_k$, and 3 to ${\rm B}_k$ and ${\rm D}_k$.
Then the orientation of the octahedral triangulation induced by this ordering satisfies the edge-orientation consistency.

\begin{figure}[h]
\centering
\includegraphics[scale=1]{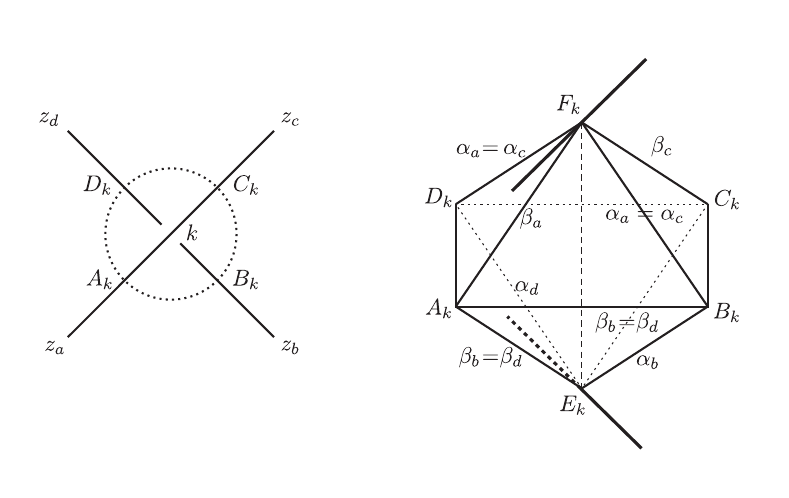}
 \caption{Labelings of non-horizontal edges}\label{labeling}
\end{figure}

\begin{obs}\label{obs} For a fixed link diagram with the octahedral triangulation, we have
$$\log\alpha_l-\log\beta_l\equiv\log z_l +A~~({\rm mod}~\pi i),$$
for all $l=1,\ldots,n$, where $n$ is the number of sides of the diagram and 
$A$ is a complex constant number independent of $l$.
\end{obs}

\begin{proof} Applying the definition of $p\pi i$ in (\ref{pq}) to the tetrahedra
${\rm E}_k{\rm F}_k{\rm A}_k{\rm B}_k$ and ${\rm E}_k{\rm F}_k{\rm C}_k{\rm B}_k$ in Figure \ref{labeling}, we have
\begin{eqnarray*}
\log z_b-\log z_a&\equiv&(\log\alpha_b-\log\beta_b)-(\log\alpha_a-\log\beta_a)~~({\rm mod}~\pi i),\\
\log z_b-\log z_c&\equiv&(\log\alpha_b-\log\beta_b)-(\log\alpha_c-\log\beta_c)~~({\rm mod}~\pi i).
\end{eqnarray*}
Note that these equations hold for all tetrahedra in the triangulation. 
Therefore, by letting $A=(\log\alpha_a-\log\beta_a)-\log z_a$, we complete the proof.

\end{proof}

Now we consider the three cases in Figure \ref{cases}. 
For $m=a,b,c,d$, let $\sigma_l^m$ be the sign of the tetrahedron between the sides $z_l$ and $z_m$,
and $u_l^m$ be the shape parameter of the tetrahedron assigned to the horizontal edge.
We put $\tau_l^m=1$ when $z_l$ is the numerator of $(u_l^m)^{\sigma_l^m}$ and $\tau_l^m=-1$ otherwise.
We also define $p_l^m$ and $q_l^m$ so that 
$\sigma_l^m[(u_l^m)^{\sigma_l^m};p_l^m,q_l^m]$ becomes the element of 
$\widehat{\mathcal{P}}(\mathbb{C})$ corresponding to the tetrahedron.
By definition, we know
\begin{equation}
u_l^a=\frac{z_l}{z_a},~u_l^b=\frac{z_b}{z_l},~u_l^c=\frac{z_l}{z_c},~u_l^d=\frac{z_d}{z_l}.\label{def_u}
\end{equation}

In the case (a) of Figure \ref{cases}, we have
$$\sigma_l^a=1,~\sigma_l^b=1,~\sigma_l^c=-1,~\sigma_l^d=-1~\text{ and }~
\tau_l^a=1,~\tau_l^b=-1,~\tau_l^c=-1,~\tau_l^d=1.$$
Using the equation (\ref{pq}) and Figure \ref{case a}, we decide $p_l^m$ and $q_l^m$ as follows:

\begin{figure}[h]
\centering
\includegraphics[scale=0.8]{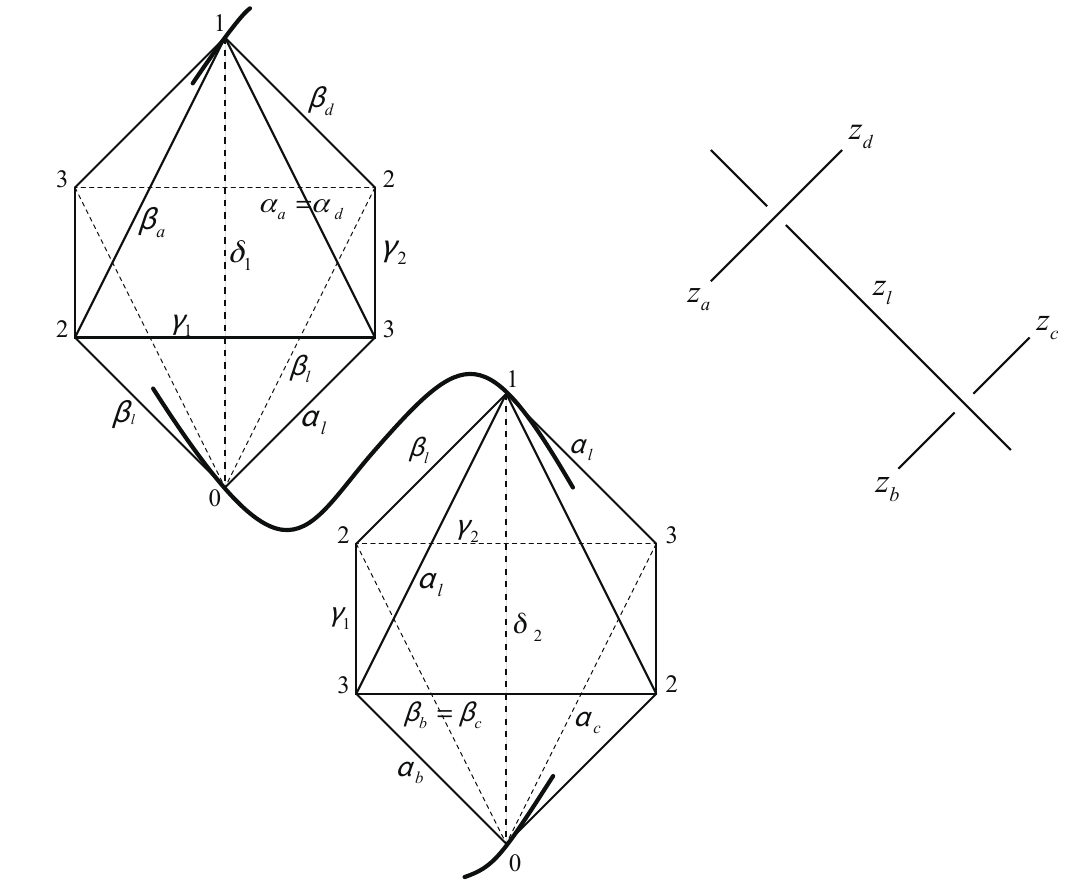}
 \caption{Case (a) of Figure \ref{cases}}\label{case a}
\end{figure}

\begin{eqnarray}
      \left\{\begin{array}{ll}
      \log\frac{z_l}{z_a}+p_l^a\pi i=\log\alpha_l+\log\beta_a-\log\beta_l-\log\alpha_a,\\
      \log\frac{z_b}{z_l}+p_l^b\pi i=\log\alpha_b+\log\beta_l-\log\beta_b-\log\alpha_l,\\
      \log\frac{z_c}{z_l}+p_l^c\pi i=\log\alpha_c+\log\beta_l-\log\beta_c-\log\alpha_l,\\
      \log\frac{z_l}{z_d}+p_l^d\pi i=\log\alpha_l+\log\beta_d-\log\beta_l-\log\alpha_d,
      \end{array}\right.\label{case a p}
\end{eqnarray}

\begin{eqnarray}
      \left\{\begin{array}{ll}
      -\log(1-\frac{z_l}{z_a})+q_l^a\pi i=\log\beta_l+\log\alpha_a-\log\gamma_1-\log\delta_1,\\
      -\log(1-\frac{z_b}{z_l})+q_l^b\pi i=\log\beta_b+\log\alpha_l-\log\gamma_1-\log\delta_2,\\
      -\log(1-\frac{z_c}{z_l})+q_l^c\pi i=\log\beta_c+\log\alpha_l-\log\gamma_2-\log\delta_2,\\
      -\log(1-\frac{z_l}{z_d})+q_l^d\pi i=\log\beta_l+\log\alpha_d-\log\gamma_2-\log\delta_1.
      \end{array}\right.\label{case a q}
\end{eqnarray}

In the case (b) of Figure \ref{cases}, we have
$$\sigma_l^a=-1,~\sigma_l^b=1,~\sigma_l^c=-1,~\sigma_l^d=1~\text{ and }~
\tau_l^a=\tau_l^b=\tau_l^c=\tau_l^d=-1.$$
Using the equation (\ref{pq}) and Figure \ref{case b}, we decide $p_l^m$ and $q_l^m$ as follows:

\begin{figure}[h]
\centering
\includegraphics[scale=0.8]{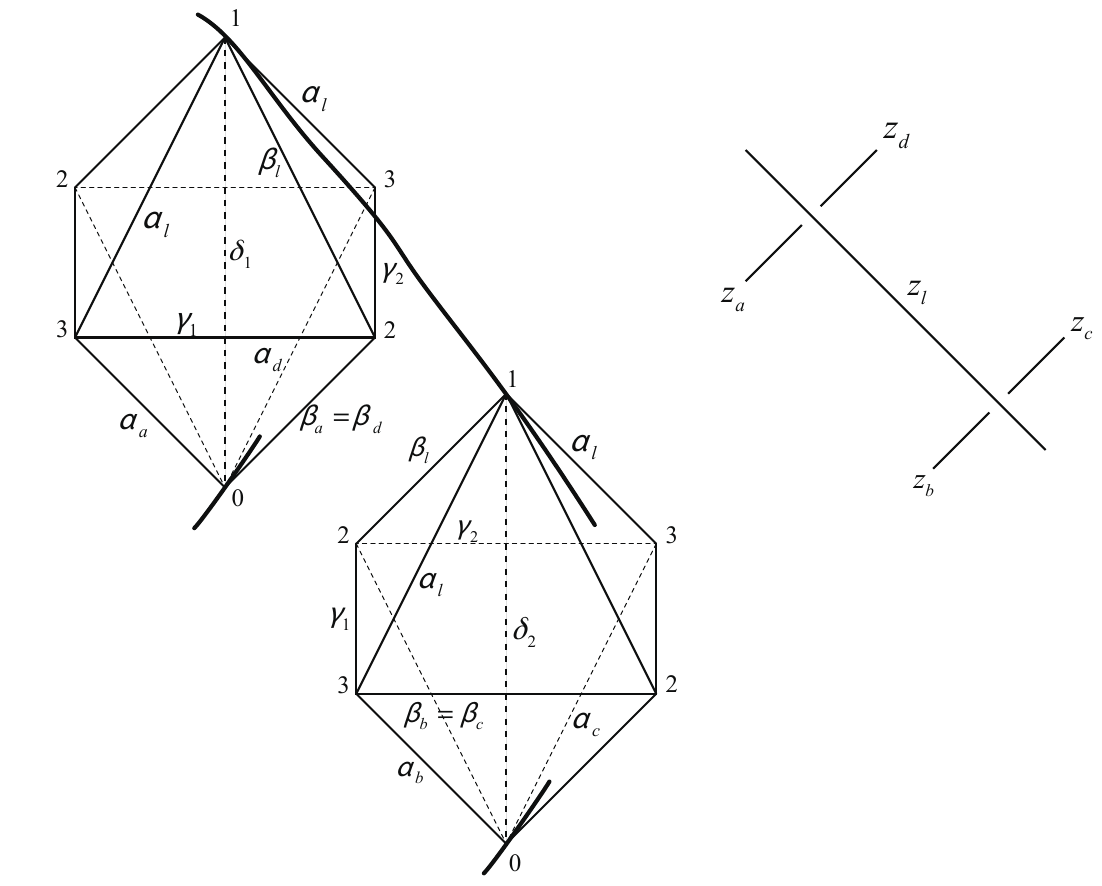}
 \caption{Case (b) of Figure \ref{cases}}\label{case b}
\end{figure}

\begin{eqnarray}
      \left\{\begin{array}{ll}
      \log\frac{z_a}{z_l}+p_l^a\pi i=\log\alpha_a+\log\beta_l-\log\beta_a-\log\alpha_l,\\
      \log\frac{z_b}{z_l}+p_l^b\pi i=\log\alpha_b+\log\beta_l-\log\beta_b-\log\alpha_l,\\
      \log\frac{z_c}{z_l}+p_l^c\pi i=\log\alpha_c+\log\beta_l-\log\beta_c-\log\alpha_l,\\
      \log\frac{z_d}{z_l}+p_l^d\pi i=\log\alpha_d+\log\beta_l-\log\beta_d-\log\alpha_l,
      \end{array}\right.\label{case b p}
\end{eqnarray}

\begin{eqnarray}
      \left\{\begin{array}{ll}
      -\log(1-\frac{z_a}{z_l})+q_l^a\pi i=\log\beta_a+\log\alpha_l-\log\gamma_1-\log\delta_1,\\
      -\log(1-\frac{z_b}{z_l})+q_l^b\pi i=\log\beta_b+\log\alpha_l-\log\gamma_1-\log\delta_2,\\
      -\log(1-\frac{z_c}{z_l})+q_l^c\pi i=\log\beta_c+\log\alpha_l-\log\gamma_2-\log\delta_2,\\
      -\log(1-\frac{z_d}{z_l})+q_l^d\pi i=\log\beta_d+\log\alpha_l-\log\gamma_2-\log\delta_1.
      \end{array}\right.\label{case b q}
\end{eqnarray}

In the case (c) of Figure \ref{cases}, we have
$$\sigma_l^a=1,~\sigma_l^b=-1,~\sigma_l^c=1,~\sigma_l^d=-1~\text{ and }~
\tau_l^a=\tau_l^b=\tau_l^c=\tau_l^d=1.$$
Using the equation (\ref{pq}) and Figure \ref{case c}, we decide $p_l^m$ and $q_l^m$ as follows:

\begin{figure}[h]
\centering
\includegraphics[scale=0.8]{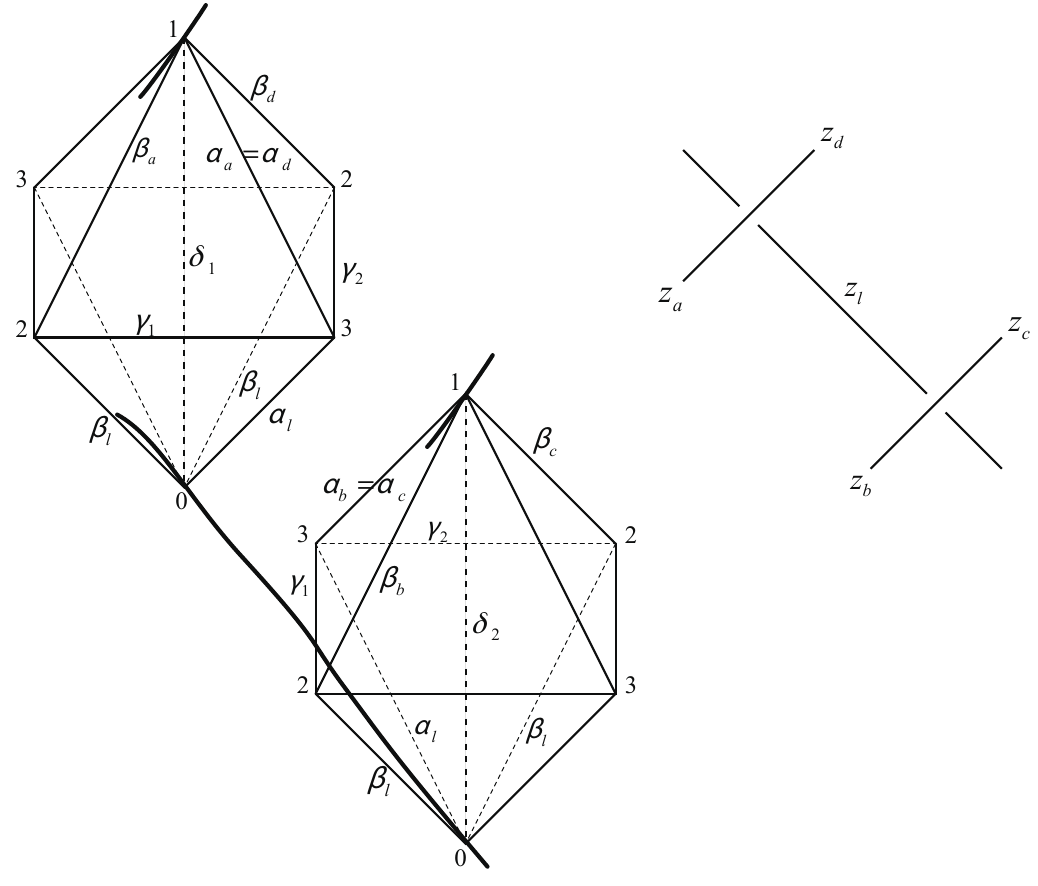}
 \caption{Case (c) of Figure \ref{cases}}\label{case c}
\end{figure}

\begin{eqnarray}
      \left\{\begin{array}{ll}
      \log\frac{z_l}{z_a}+p_l^a\pi i=\log\alpha_l+\log\beta_a-\log\beta_l-\log\alpha_a,\\
      \log\frac{z_l}{z_b}+p_l^b\pi i=\log\alpha_l+\log\beta_b-\log\beta_l-\log\alpha_b,\\
      \log\frac{z_l}{z_c}+p_l^c\pi i=\log\alpha_l+\log\beta_c-\log\beta_l-\log\alpha_c,\\
      \log\frac{z_l}{z_d}+p_l^d\pi i=\log\alpha_l+\log\beta_d-\log\beta_l-\log\alpha_d,
      \end{array}\right.\label{case c p}
\end{eqnarray}

\begin{eqnarray}
      \left\{\begin{array}{ll}
      -\log(1-\frac{z_l}{z_a})+q_l^a\pi i=\log\beta_l+\log\alpha_a-\log\gamma_1-\log\delta_1,\\
      -\log(1-\frac{z_l}{z_b})+q_l^b\pi i=\log\beta_l+\log\alpha_b-\log\gamma_1-\log\delta_2,\\
      -\log(1-\frac{z_l}{z_c})+q_l^c\pi i=\log\beta_l+\log\alpha_c-\log\gamma_2-\log\delta_2,\\
      -\log(1-\frac{z_l}{z_d})+q_l^d\pi i=\log\beta_l+\log\alpha_d-\log\gamma_2-\log\delta_1.
      \end{array}\right.\label{case c q}
\end{eqnarray}

Note that 
\begin{eqnarray*}
\sigma_l^m=\sigma_m^l,~ \tau_l^m=-\tau_m^l,~ u_l^m=u_m^l, ~p_l^m=p_m^l, ~q_l^m=q_m^l\text{ and}\\
\sigma_l^m[(u_l^m)^{\sigma_l^m};p_l^m,q_l^m]=\sigma_m^l[(u_m^l)^{\sigma_m^l};p_m^l,q_m^l]\in\widehat{\mathcal{P}}(\mathbb{C}).
\end{eqnarray*}
If we put the element\footnote{
The element has the coefficient $\frac{1}{2}$ because all tetrahedra appear twice in the summation.}
 $\frac{1}{2}\sum_{l,m}\sigma_l^m[(u_l^m)^{\sigma_l^m};p_l^m,q_l^m]
\in\widehat{\mathcal{P}}(\mathbb{C})$ 
corresponding to the triangulation of $\mathbb{S}^3\backslash (L\cup\{\pm\infty\})$, 
the potential function defined in Section \ref{sec2} can be expressed by the following way:
$$V(z_1,\ldots,z_n)=\frac{1}{2}\sum_{l,m}\sigma_l^m\li\left((u_l^m)^{\sigma_l^m}\right).$$
By direct calculation, we obtain
\begin{equation}\label{partial}
z_l\frac{\partial V}{\partial z_l}=-\sum_{m=a,\ldots,d}\sigma_l^m\tau_l^m\log(1-(u_l^m)^{\sigma_l^m})
\end{equation}
for all $l=1,\ldots,n$.

Recall that we use the notation $r_k$ instead of $r_{j,k}$, which was defined in (\ref{def_r}).

\begin{lem}\label{lem42} 
For all $l=1,\ldots,n$, we have
$$r_l\pi i=-\sum_{m=a,\dots,d}\sigma_l^m \tau_l^m q_l^m\pi i.$$
\end{lem}

\begin{proof}
In the case (a) of Figure \ref{cases}, using (\ref{def_r}), (\ref{case a q}), (\ref{partial}), $\alpha_a=\alpha_d$ and $\beta_b=\beta_c$,
 we can directly calculate the following:
\begin{eqnarray}
r_l\pi i&=&z_l\frac{\partial V}{\partial z_l}=-\sum_{m=a,\ldots,d}\sigma_l^m\tau_l^m\log(1-(u_l^m)^{\sigma_l^m})\label{abc}\\
&=&-q_l^a\pi i+q_l^b\pi i-q_l^c\pi i+q_l^d\pi i.\nonumber
\end{eqnarray}
The cases (b) and (c) of Figure \ref{cases} also can be proved by the direct calculation using (\ref{case b q}) and (\ref{case c q}).

\end{proof}

\begin{cor}\label{cor4} 
For all possible $l$ and $m$, we have
\begin{equation}\label{q2pi}
\frac{1}{2}\sum_{l,m}\sigma_l^m q_l^m\pi i\log(u_l^m)^{\sigma_l^m}\equiv-\sum_{l=1}^n r_l\pi i\log z_l~~({\rm mod}~2\pi^2).
\end{equation}
\end{cor}

\begin{proof} 
Note that $q_l^m$ is an integer. Using (\ref{def_u}) and Lemma \ref{lem42}, we can directly calculate
\begin{eqnarray*}
\frac{1}{2}\sum_{l=1}^n\sum_{m=a,\ldots,d}\sigma_l^m q_l^m\pi i\log (u_l^m)^{\sigma_l^m}
&\equiv&\sum_{l=1}^n\left(\sum_{m=a,\ldots,d}\sigma_l^m\tau_l^m q_l^m\pi i\right)\log z_l~~({\rm mod}~2\pi^2)\\
&=&-\sum_{l=1}^n r_l\pi i\log z_l.
\end{eqnarray*}
\end{proof}

\begin{lem}\label{lem44}
For all possible $l$ and $m$, we have
$$\frac{1}{2}\sum_{l,m}\sigma_l^m\log\left(1-(u_l^m)^{\sigma_l^m}\right)\left(\log(u_l^m)^{\sigma_l^m}+p_l^m\pi i\right)
\equiv-\sum_{l=1}^n r_l\pi i\log z_l~~({\rm mod}~2\pi^2).$$
\end{lem}

\begin{proof} 
Substituting the term $\left(\log(u_l^m)^{\sigma_l^m}+p_l^m\pi i\right)$ to the summation of $\pm(\log\alpha_k-\log\beta_k)$ terms
by applying (\ref{case a p}) or (\ref{case b p}) or (\ref{case c p}) and (\ref{abc}), we can verify
\begin{eqnarray*}
\lefteqn{\frac{1}{2}\sum_{l,m}\sigma_l^m\log\left(1-(u_l^m)^{\sigma_l^m}\right)\left(\log(u_l^m)^{\sigma_l^m}+p_l^m\pi i\right)}\\
&&=\sum_{l=1}^n\left(\sum_{m=a,\ldots,d}\sigma_l^m\tau_l^m\log(1-(u_l^m)^{\sigma_l^m})\right)(\log\alpha_l-\log\beta_l)\\
&&=-\sum_{l=1}^n r_l\pi i(\log\alpha_l-\log\beta_l).
\end{eqnarray*}
Note that $r_l$ is an even integer and
$$z_j\frac{\partial\li(z_j/z_k)}{\partial z_j}+z_k\frac{\partial\li(z_j/z_k)}{\partial z_k}=
-\log(1-\frac{z_j}{z_k})+\log(1-\frac{z_j}{z_k})=0$$
implies $\displaystyle\sum_{l=1}^n r_{l}\pi i=0.$
By using Observation \ref{obs} and the above property, we have
\begin{equation*}
-\sum_{l=1}^n r_l\pi i(\log\alpha_l-\log\beta_l)\equiv-\sum_{l=1}^n r_l\pi i(\log z_l+A)=-\sum_{l=1}^n r_l\pi i\log z_l~~({\rm mod}~2\pi^2).
\end{equation*}
\end{proof}

Combining (\ref{formula}), Corollary \ref{cor4} and Lemma \ref{lem44}, we prove (\ref{V1}) as follows:
\begin{eqnarray*}
\lefteqn{i(\vol(\rho_{\bold z})+i\,\cs(\rho_{\bold z}))
\equiv\widehat{L}\left(\frac{1}{2}\sum_{l,m}\sigma_l^m[(u_l^m)^{\sigma_l^m};p_l^m,q_l^m]\right)}\\
&&=\frac{1}{2}\sum_{l,m}\sigma_l^m\left(\li\left((u_l^m)^{\sigma_l^m}\right)-\frac{\pi^2}{6}\right)
+\frac{1}{4}\sum_{l,m}\sigma_l^m q_l^m\pi i\log\left(u_l^m\right)^{\sigma_l^m}\\
&&~~+\frac{1}{4}\sum_{l,m}\sigma_l^m\log\left(1-\left(u_l^m\right)^{\sigma_l^m}\right)
\left(\log\left(u_l^m\right)^{\sigma_l^m}+p_l^m\pi i\right)\\
&&\equiv V(z_1,\ldots,z_n)-\sum_{l=1}^n r_l\pi i \log z_l=V_0({\bold z})\modulo.
\end{eqnarray*}

The existence of a solution $\bold{z_{\infty}}$ satisfying $\vol(\rho_{\bold{z_{\infty}}})=\vol(L)$ 
follows from Theorem 1.1 of \cite{Tillmann13}.
Although this theorem was proved for closed manifolds, 
it is still true for our case because Thurston's spinning construction and all the other steps of the proof are valid.
From Thurston-Gromov-Goldman rigidity (Theorem 7.1 in \cite{Francaviglia04}) and (\ref{V1}),
we know $\rho_{\bold{z_{\infty}}}$ is the discrete and faithful representation
and (\ref{V2}) holds.
One minor remark is that Theorem 1.1 of \cite{Tillmann13} 
considered parameter space of Thurston's gluing equations (without completeness condition),
so $\bold{z_{\infty}}$ lies in the parameter space. 
However, because $\rho_{\bold{z_{\infty}}}$ is discrete and faithful,
it is boundary-parabolic and $\bold{z_{\infty}}$ also satisfies the completeness condition.
Therefore $\bold{z_{\infty}}\in \mathcal{S}$ and the path component of $\mathcal{S}$ containing $\bold{z_{\infty}}$ is $\mathcal{S}_0$.

\section{Examples of the twist knots}\label{sec5}

Let $T_n$ ($n\geq 1$) be the twist knot with $n+3$ crossings in Figure \ref{twistknots}. 
For example, $T_1$ is the figure-eight knot $4_1$ and $T_2$ is the $5_2$ knot.
In this section, we show an application of Theorem \ref{thm1} to the twist knot $T_n$ and several numerical results.

\begin{figure}[h]
\centering
\includegraphics[scale=0.9]{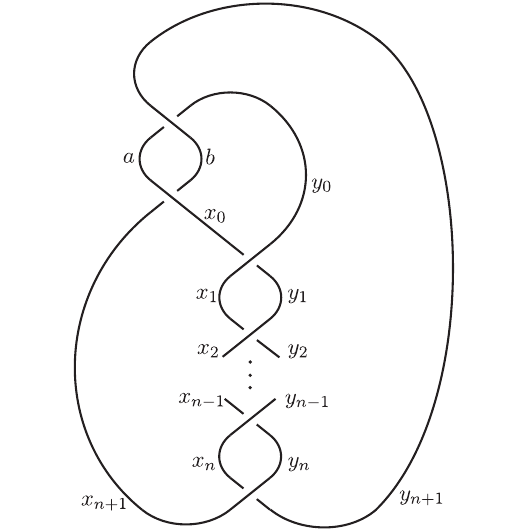}
 \caption{Twist knots $T_n$}\label{twistknots}
\end{figure}

We assign variables $a,b,x_0,\ldots,x_{n+1},y_0,\ldots,y_{n+1}$ to sides of Figure \ref{twistknots}. Then
the potential function becomes
\begin{eqnarray*}
\lefteqn{V(T_n;a,b,x_0,\ldots,x_{n+1},y_0,\ldots,y_{n+1})}\\
&&=\left\{\li(\frac{y_0}{b})-\li(\frac{y_0}{y_{n+1}})+\li(\frac{a}{y_{n+1}})-\li(\frac{a}{b})\right\}\\
&&~~~+\left\{\li(\frac{b}{x_0})-\li(\frac{b}{a})+\li(\frac{x_{n+1}}{a})-\li(\frac{x_{n+1}}{x_0})\right\}\\
&&~~~+\sum_{k=0}^{n}\left\{\li(\frac{y_{k+1}}{x_{k+1}})-\li(\frac{y_{k+1}}{y_k})+\li(\frac{x_k}{y_k})-\li(\frac{x_k}{x_{k+1}})\right\}.
\end{eqnarray*}
We abbreviate the notation of this function to $V(T_n)$.

Finding the whole solution set of the hyperbolicity equations
\begin{equation*}
\mathcal{H}(T_n):=\left\{\left.\exp(z\frac{\partial V(T_n)}{\partial z})=1~\right|~z=a,b,x_0,\ldots,x_{n+1},y_0,\ldots,y_{n+1}\right\}
\end{equation*}
is neither easy nor useful. Instead, we can obtain enough solutions by fixing certain numbers,
say $a=2$, $b=-1$ and $y_{n+1}=1$. Then, from $\exp(a\frac{\partial V(T_n)}{\partial a})=1$, we find $x_{n+1}=3$. 

If we denote $x_0=t$, then we can express all the other variables using $t$ as follows:

from $\exp(b\frac{\partial V(T_n)}{\partial b})=1$, we find $y_0=1+\frac{2}{t}$. 
Also, from $\exp(x_0\frac{\partial V(T_n)}{\partial x_0})=1$ and $\exp(y_0\frac{\partial V(T_n)}{\partial y_0})=1$,
we find $x_1=\frac{t(t+2)}{t^2-4t+8}$ and $y_1=\frac{4}{t}$.

\begin{figure}[h]
\centering
\includegraphics[scale=0.9]{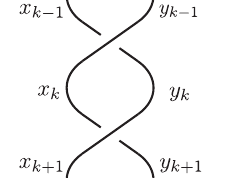}
 \caption{Crossings of the twist knot ($k=1,\ldots,n$)}\label{crossings}
\end{figure}

For $k=1,\ldots,n$, the equations $\exp(x_{k}\frac{\partial V(T_n)}{\partial x_{k}})=1$ 
and $\exp(y_{k}\frac{\partial V(T_n)}{\partial y_{k}})=1$ of Figure \ref{crossings} induce the following recursive formulas
$$x_{k+1}=\frac{x_k y_k}{-x_{k-1}+x_k+y_k},~y_{k+1}=x_k+y_k-\frac{x_k y_k}{y_{k-1}}.$$
They enable us to express all the variables in rational polynomials of $t$. 
Table \ref{table3} shows $x_{k}$ and $y_{k}$ in $t$ for $k=0,\ldots,5$.

\begin{table}[h]
{
\begin{tabular}{|c|c|}\hline
  $x_0$ & $\displaystyle t$ \\\hline
  $y_0$& $\displaystyle \frac{2+t}{t}$\\\hline
  $x_1$& $\displaystyle \frac{2t+t^2}{8-4t+t^2}$ \\\hline
  $y_1$&$\displaystyle  \frac{4}{t}$\\\hline
  $x_2$&$\displaystyle \frac{-4t}{-16+16t-7t^2+t^3}$\\\hline
  $y_2$& $\displaystyle \frac{32-16t+2t^2+t^3}{8t-4t^2+t^3}$\\\hline
  $x_3$&$\displaystyle \frac{32t-16t^2+2t^3+t^4}{128-192t+128t^2-40t^3+5t^4}$\\\hline
  $y_3$& $\displaystyle \frac{-64+64t-24t^2+t^4}{-16t+16t^2-7t^3+t^4}$\\\hline
  $x_4$&$\displaystyle \frac{-64t+64t^2-24t^3+t^5}{-256+512t-464t^2+224t^3-57t^4+6t^5}$\\\hline
  $y_4$& $\displaystyle \frac{512-768t+480t^2-112t^3-6t^4+5t^5}{128t-192t^2+128t^3-40t^4+5t^5}$ \\\hline
  $x_5$&$\displaystyle \frac{512t-768t^2+480t^3-112t^4-6t^5+5t^6}{2048-5120t+5888t^2-3840t^3+1480t^4-316t^5+29t^6}$\\\hline
  $y_5$&$\displaystyle \frac{-1024+2048t-1792t^2+768t^3-124t^4-16t^5+6t^6}{-256t+512t^2-464t^3+224t^4-57t^5+6t^6}$\\\hline
%  $x_6$& $\displaystyle \frac{-1024t+2048t^2-1792t^3+768t^4-124t^5-16t^6+6t^7}{-4096+12288t-17152t^2+14080t^3-7264t^4+2336t^5-431t^6+35t^7}$\\\hline
%  $y_6$&$\displaystyle \frac{8192-20480t+23040t^2-14080t^3+4544t^4-480t^5-118t^6+29t^7}{2048t-5120t^2+5888t^3-3840t^4+1480t^5-316t^6+29t^7}$\\\hline
  \end{tabular}\caption{$x_k$ and $y_k $ for $k=0,\dots,6$}\label{table3}}
\end{table}

Furthermore, $\exp(y_{n+1}\frac{\partial V(T_n)}{\partial y_{n+1}})=1$
gives a simple relation
\begin{equation}\label{xn}
y_n=\frac{3t}{3t-4},
\end{equation}
which determines {\it the defining equation of} $t$. 
Table \ref{table1} shows these equations for $n=1,\ldots,5$.
 
\begin{table}[h]
{\begin{tabular}{|c|c|}\hline
  $n$ & Defining equation of $t$\\\hline
  1 & $16-12t+3t^2=0$\\\hline
  2 & $-64+80t-40t^2+7t^3=0$\\\hline
  3 & $256-448t+336t^2-120t^3+17t^4=0$\\\hline
  4 & $-2048+4608t-4608t^2+2464t^3-696t^4+82t^5=0$\\\hline
  5 & $4096-11264t+14080t^2-9984t^3+4192t^4-980t^5+99t^6=0$\\\hline
\end{tabular}\caption{Defining equation (\ref{xn}) of $t$ for $n=1,\dots,5$}\label{table1}}
\end{table}

We checked all the solutions $t$ of the defining equation (\ref{xn}) in Table \ref{table1} satisfy the equations
$\exp(x_n\frac{\partial V(T_n)}{\partial x_n})=1$, $\exp(y_n\frac{\partial V(T_n)}{\partial y_n})=1$ and
$\exp(x_{n+1}\frac{\partial V(T_n)}{\partial x_{n+1}})=1$.\footnote{
As a matter of fact, checking only two of them is enough. This is because,
from the fact $\sum_{z}\frac{\partial V(T_n)(...,z,...)}{\partial z}=0$, one of the equations of $\mathcal{H}(T_n)$ can be deduced from
the others.} 
Therefore, all the solutions $t$ of the defining equation determine the solutions of $\mathcal{H}(T_n)$.
We denote the corresponding representation of $t$ by 
$$\rho(T_n)(t):\pi_1(\mathbb{S}^3\backslash T_n)\longrightarrow{\rm PSL}(2,\mathbb{C}).$$
Then Table \ref{table2} shows the values of $t$ and the corresponding complex volumes of $\rho(T_n)(t)$ for $n=1,\ldots,5$.
Note that, when $n=2$, the values in Table \ref{table2} coincide with the result of the $5_2$ knot in Example 6.16 of \cite{Zickert09}.

\begin{table}[h]
{\begin{tabular}{|c|l|l|}\hline
  $n$ &\hspace{2.2cm} $t$ & $V_0(T_n)(t)\equiv i(\vol(\rho(T_n)(t))+i\,\cs(\rho(T_n)(t)))$\\\hline
  1 & $t=2+1.1547...i$ &\hspace{1cm} $i(2.0299...+0\,i)$\\
    &$t=2-1.1547...i$ &\hspace{1cm} $i(-2.0299...+0\,i)$\\\hline
  2 & $t=1.4587...+1.0682...i$ &\hspace{1cm} $i(2.8281...+3.0241...i)$\\
    & $t=1.4587...-1.0682...i$ &\hspace{1cm} $i(-2.8281...+3.0241...i)$\\
    & $t=2.7969...$ &\hspace{1cm} $i(0-1.1135...i)$\\\hline
  3 & $t=1.2631...+1.0347...i$ &\hspace{1cm} $i(3.1640...+6.7907...i)$\\
    & $t=1.2631...-1.0347...i$ &\hspace{1cm} $i(-3.1640...+6.7907...i)$\\
    & $t=2.2664...+0.7158...i$ &\hspace{1cm} $i(1.4151...+0.2110...i)$\\
    & $t=2.2664...-0.7158...i$ &\hspace{1cm} $i(-1.4151...+0.2110...i)$\\\hline
  4 & $t=1.1713...+1.0202...i$ &\hspace{1cm} $i(3.3317...+10.9583...i)$\\
    & $t=1.1713...-1.0202...i$ &\hspace{1cm} $i(-3.3317...+10.9583...i)$\\
    & $t=1.8097...+0.9073...i$ &\hspace{1cm} $i(2.2140...+1.8198...i)$\\
    & $t=1.8097...-0.9073...i$ &\hspace{1cm} $i(-2.2140...+1.8198...i)$\\
    & $t=2.5257...$ &\hspace{1cm} $i(0-0.8822...i)$\\\hline
  5 & $t=1.1208...+1.0129...i$ &\hspace{1cm} $i(3.4272...+15.3545...i)$\\
    & $t=1.1208...-1.0129...i$ &\hspace{1cm} $i(-3.4272...+15.3545...i)$\\
    & $t=1.5498...+0.9676...i$ &\hspace{1cm} $i(2.6560...+4.6428...i)$\\
    & $t=1.5498...-0.9676...i$ &\hspace{1cm} $i(-2.6560...+4.6428...i)$\\
    & $t=2.2789...+0.4876...i$ &\hspace{1cm} $i(1.1087...-0.2581...i)$\\
    & $t=2.2789...-0.4876...i$ &\hspace{1cm} $i(-1.1087...-0.2581...i)$\\\hline
\end{tabular}\caption{Complex volumes of $\rho(T_n)(t)$ for $n=1,\ldots,5$}\label{table2}}
\end{table}

Note that, from the well-known property (see Proposition 4.8 of \cite{Tillmann13} for example), the value $V_0(T_n)(t)$ with the maximal
imaginary part is the complex volume $i(\vol(T_n)+i\,\cs(T_n))$ of the hyperbolic knot $T_n$.
We placed them at the top in Table \ref{table2}.

We finally remark that the calculation method in this section also works for $n>5$ and
finding complete solutions of $\mathcal{H}(T_n)$ for small $n$ (and probably for all $n$) is possible.
However, all the values of $V_0(T_n)$ evaluated at the complete solutions lie in 
Table \ref{table2} for $n\leq 5$ (and probably do for any $n>5$ too). 
Therefore, our restricted solutions are general enough for calculating complex volumes of twist knots.

\vspace{5mm}
\begin{ack}
  The authors appreciate Yunhi Cho and Jun Murakami for discussions and suggestions on this work.
  The second author was supported by the Korea Research Foundation(KRF)grant
funded by the Korea government(MEST)(No. 2013R1A1A2005861).
\end{ack}

{
%\sc
\begin{flushleft}
  Department of Mathematics, Korea Institute for Advanced Study (KIAS), Seoul 130-722, Republic of Korea\\
  \vspace{0.4cm}
  Department of Mathematical Sciences, Seoul National University, Seoul 151-747, Republic of Korea\\
    \vspace{0.4cm}
  Center for Geometry and Physics, Institute for Basic Science (IBS), Pohang 790-784, Republic of Korea\\
    \vspace{0.4cm}
E-mail: dol0425@gmail.com\\
\mbox{\phantom{E-mail: }hyukkim@snu.ac.kr}\\
\mbox{\phantom{E-mail: }ryeona17@ibs.re.kr}

\end{flushleft}}
\end{document}